\documentclass[12pt]{amsart}
\usepackage{amscd,amssymb}
\usepackage[arrow,matrix]{xy}
\usepackage[colorlinks,plainpages,backref,urlcolor=blue]{hyperref}

\newtheorem{theorem}[subsection]{Theorem}
\newtheorem{lemma}[subsection]{Lemma}
\newtheorem{prop}[subsection]{Proposition}
\newtheorem{corollary}[subsection]{Corollary}

\newtheorem{thm}{Theorem}

\theoremstyle{definition}
\newtheorem{remark}[subsection]{Remark}

\newtheorem{example}[subsection]{Example}
\newtheorem*{ack}{Acknowledgments}

\numberwithin{equation}{section}
\newcommand{\X}{{\mathcal X}}

\newcommand{\cH}{{\mathcal H}}
\newcommand{\cG}{{\mathcal G}}

\newcommand{\M}{{\mathcal M}}

\newcommand{\A}{{\mathcal A}}
\newcommand{\OO}{{\mathcal O}}

\newcommand{\W}{{\mathcal W}}

\newcommand{\V}{{\mathcal V}}
\newcommand{\R}{{\mathcal R}}
\newcommand{\TT}{{\mathcal T}}

\renewcommand{\SS}{{\mathcal S}}

\newcommand{\h}{\mathfrak {h}}
\newcommand{\m}{\mathfrak {m}}
\newcommand{\n}{\mathfrak {n}}  
\newcommand{\bb}{\mathfrak {b}}
\newcommand{\s}{\mathfrak {s}}
\renewcommand{\sp}{\mathfrak {sp}}

\newcommand{\T}{\mathbb{T}}
\newcommand{\Z}{\mathbb{Z}}

\newcommand{\Q}{\mathbb{Q}}
\newcommand{\RR}{\mathbb{R}}
\newcommand{\C}{\mathbb{C}}

\newcommand{\PP}{\mathbb{P}}
\newcommand{\FF}{\mathbb{F}}

\newcommand{\bL}{\mathbb{L}}
\newcommand{\bS}{\mathbb{S}}

\newcommand{\G}{\Gamma}

\DeclareMathOperator{\Hom}{Hom}
\DeclareMathOperator{\rank}{rank}
\DeclareMathOperator{\im}{im}

\DeclareMathOperator{\coker}{coker}
\DeclareMathOperator{\spn}{span}
\DeclareMathOperator{\id}{id}

\DeclareMathOperator{\gr}{gr}

\DeclareMathOperator{\ab}{ab}
\DeclareMathOperator{\ann}{ann}
\DeclareMathOperator{\abf}{abf}
\DeclareMathOperator{\ad}{ad}

\DeclareMathOperator{\Zero}{Zero}
\DeclareMathOperator{\Sym}{Sym}

\DeclareMathOperator{\Lie}{Lie}

\DeclareMathOperator{\aut}{Aut}
\DeclareMathOperator{\out}{Out}
\DeclareMathOperator{\orb}{orb}
\DeclareMathOperator{\oder}{ODer}

\newcommand{\surj}{\twoheadrightarrow}
\newcommand{\isom}{\xrightarrow{\,\simeq\,}}

\begin{document}

\title [Arithmetic symmetry and finiteness properties]
{Arithmetic group symmetry and finiteness properties of Torelli groups}

\author[Alexandru Dimca]{Alexandru Dimca$^1$}
\address{Institut Universitaire de France et Laboratoire J.A. Dieudonn\'e, UMR du CNRS 7351,
                 Universit\'e de Nice Sophia-Antipolis,
                 Parc Valrose,
                 06108 Nice Cedex 02,
                 France}
\email{dimca@unice.fr}

\author[Stefan Papadima]{Stefan Papadima$^2$}
\address{Simion Stoilow Institute of Mathematics, 
P.O. Box 1-764,
RO-014700 Bucharest, Romania}
\email{Stefan.Papadima@imar.ro}

\thanks{$^1$ Partially supported by the French-Romanian Programme LEA Math-Mode
and  ANR-08-BLAN-0317-02 (SEDIGA)} 
\thanks{$^2$ Partially supported by the French-Romanian Programme LEA Math-Mode
and PN-II-ID-PCE-2011-3-0288, grant 132/05.10.2011}



\begin{abstract}
We examine groups whose resonance varieties, characteristic varieties and Sigma-invariants
have a natural arithmetic group symmetry, and we explore implications on various finiteness properties
of subgroups. We compute resonance varieties, characteristic varieties and Alexander polynomials
of Torelli groups, and we show that all subgroups containing the Johnson kernel have finite 
first Betti number, when the genus is at least $4$. We also prove that, in this range, the
$I$-adic completion of the Alexander invariant is finite-dimensional, and the Kahler property
for the Torelli group implies the finite generation of the Johnson kernel.
\end{abstract}

\maketitle

\tableofcontents

\section{Introduction} \label{sec:intro}

The mapping class group $\G_g$, consisting of the isotopy classes of orientation--preserving
homeomorphisms of a closed, oriented, genus $g$ surface $\Sigma_g$, is an ubiquitous character on 
the mathematical stage. The {\em Torelli} subgroup $T_g$, that is, the kernel of the action on the first $\Z$--homology
of $\Sigma_g$, defined by the exact sequence
\begin{equation}
\label{eq:deftintro}
1\to T_g \longrightarrow \G_g \stackrel{p}{\longrightarrow} Sp_g(\Z)\to 1\, ,
\end{equation}
brings into play the arithmetic integral symplectic group $Sp_g(\Z)$. (Since $T_g$ is trivial for $g\le 1$, 
we will assume $g\ge 2$.) It is a celebrated theorem of Poincar\' e \cite{Po} that $p$ is a surjection. The
{\em Johnson kernel} $K_g$, analyzed by Dennis Johnson in a remarkable series of papers \cite{J0}--\cite{J3},
is the subgroup of $T_g$ generated by Dehn twists about  separating curves on $\Sigma_g$. Our main result in
this note is the following.

\begin{thm}
\label{thm:johnkf}
For $g\ge 4$, the vector space $H_1(K_g, \Q)$ is finite--dimensional.
\end{thm}

The question of whether the first Betti number of $K_g$ is finite has been open for many years, since Johnson's work in the early
80's, despite the attention of many experts. When we started working on this problem, the general consensus seemed to be that 
the answer is probably no; see
for instance Hain's survey \cite{H-sur} on expected infiniteness properties related to Torelli groups. In this direction, Akita's result 
\cite{Ak} guarantees the existence of an infinite Betti number of $K_g$. Biss and Farb tried to prove in \cite{BF} that $K_g$ is not
finitely generated, but a fatal error was found in their paper. As we shall see, Theorem \ref{thm:johnkf} actually holds for
all subgroups of $T_g$ containing $K_g$, which answers Problem 5.3 from Farb's list \cite{F}.

Another important problem concerns the Kahler property: is $T_g$ a Kahler group (that is,
the fundamental group of a compact Kahler manifold)? The answer is negative, in low genus:
$T_2$ is not finitely generated, as proved by McCullough and Miller \cite{MM}, and $T_3$
violates Dennis Sullivan's \cite{Su} $1$-formality property of Kahler groups (established by 
Deligne, Griffiths, Morgan and Sullivan in \cite{DGMS}), as shown by Hain in \cite{H}. Theorem 4.9 from \cite{F} states that
$T_g$ is not a Kahler group. Since the proof assumed infinite generation of $K_g$,
the Kahler group problem is open as well, for $g\ge 4$. Our second main result connects 
the Kahler problem for Torelli groups with the finite generation question for Johnson kernels.

\begin{thm}
\label{thm:kintro}
If $T_g$ is a Kahler group, the Johnson kernel $K_g$ must be finitely generated, for $g\ge 4$.
\end{thm}

\subsection{Finiteness properties of the Johnson filtration} \label{ssi1}

The group $\G_g$ may be viewed from many angles. By a result of Earle and Eells, the Eilenberg--MacLane space
$K(\G_g, 1)$ classifies oriented $\Sigma_g$--bundles. Via Heegaard splittings, $\G_g$ is intimately related with the
world of $3$--manifolds. The group $\G_g$ is the orbifold fundamental group of the (aspherical) moduli space of
genus $g$ compact Riemann surfaces, $\M_g$: $\G_g$ acts properly discontinuously, with finite isotropy, on the 
(contractible) Teichm\" uller space $\X_g$,  with quotient $\X_g /\G_g =\M_g$. Furthermore, the so-called 
(finite index, normal) level subgroup $\G_g (\ell)$ is torsion--free, and the quotient $\X_g /\G_g (\ell)$ is the
(quasi--projective) moduli space of complex curves with level $\ell$ structure, for $\ell \ge 3$. So, $\G_g$
enjoys almost every conceivable finiteness property: it is finitely presentable, it has finite virtual cohomological
dimension, and all its homology groups are finitely generated. However, little is known about finiteness properties of
infinite index subgroups, such as $T_g$.

We are going to adopt a group--theoretical viewpoint on $\G_g$, that emerges from Nielsen theory: 
$\G_g =\out^{+} (\pi_g)$, the proper outer group of automorphisms of the group $\pi_g= \pi_1(\Sigma_g)$. 
In this way, we may identify both $T_g$ and $K_g$ with the first terms in the so-called {\em Johnson filtration}
of $\Gamma_g$. This filtration was first introduced and analyzed for $\out (\FF_n)$, by Andreadakis \cite{And},  
where $\FF_n$ is the free group on $n$ generators.  

The construction goes as follows. Given a group $\pi$, denote by $\pi^s$ the subgroup generated by length 
$s$ commutators. (Here, $\pi^1=\pi$, $\pi^2$ is the derived subgroup $\pi'$, and so on.) Let 
$\gr_{\bullet} \pi= \oplus_{s\ge 1} \pi^s/\pi^{s+1}$ be the associated graded Lie algebra. The Johnson filtration 
$\{ J^s (\pi)\}_{s\ge 0}$ is the normal, descending series defined by
\begin{equation}
\label{eq=defjfilt}
J^s (\pi)= \ker \{ \out (\pi) \rightarrow \out (\pi/\pi^{s+1}) \}
\end{equation}
(with the convention that $J^0 (\pi_g)= \out^+ (\pi_g) $). 

We infer from \eqref{eq:deftintro} that $J^1 (\pi_g)=T_g$. By Teichm\" uller theory, the complex manifold 
$\X_g/T_g =K(T_g, 1)$ is the moduli space of curves with a marked symplectic basis of $H_1(\Sigma_g, \Z)$. 
Nevertheless, the (higher) finiteness properties of $T_g$ are largely a mystery. A deep theorem of Johnson \cite{J1} 
says that $T_g$ is finitely generated for $g\ge 3$, but finite presentability is open. Another deep result, due to
Bestvina, Bux and Margalit \cite{BBM}, is that $H_{3g-5}(T_g, \Z)$ is infinitely generated, for $g\ge 2$. 

A key tool in the investigation of the Johnson filtration is provided by the {\em Johnson homomorphisms},
\begin{equation}
\label{eq=defjhom}
\{ \tau_s : J^s(\pi) \longrightarrow \oder^s (\gr_{\bullet} \pi) \}_{s\ge 1} \, .
\end{equation}
They are induced by the map that associates to $\varphi \in J^s (\pi)$ and $x\in \pi^t$ the class modulo $\pi^{s+t+1}$ 
of $\varphi (x)\cdot x^{-1}$. They take values in the {\em outer derivations} of the Lie algebra $\gr_{\bullet} \pi$ that
raise degree by $s$. The construction works well for a class of groups that includes $\pi=\pi_g$ for $g\ge 2$ and 
$\pi= \FF_n$ for $n\ge 2$; see \cite{PS10} for full details.

By construction, $J^{s+1}(\pi)=\ker (\tau_s)$. Another useful property is related to the natural {\em symmetry group},
$J^0(\pi)/J^1(\pi)$,  that acts by conjugation on $\oder^s (\gr_{\bullet} \pi)$. When $\pi=\pi_g$ (respectively $\pi=\FF_n$),
this group is $Sp_g(\Z)$ (respectively $GL_n(\Z)$). The presence of this arithmetic group symmetry will be very important later on.
In particular, it opens the way for the use of classical representation theory.

When $\pi=\pi_g$ and $g\ge 3$, $\tau_1$ is Johnson's homomorphism constructed in \cite{J0}, 
$\tau_1 : T_g \rightarrow \bigwedge^3 H_1(\Sigma_g, \Z)/H_1(\Sigma_g, \Z)$. By Johnson's fundamental result from \cite{J3},
$\tau_1$ induces an isomorphism, $(T_g)_{\abf} \stackrel{\simeq}{\rightarrow} \bigwedge^3 H_1(\Sigma_g, \Z)/H_1(\Sigma_g, \Z)$;
here, $T_{\abf}$ denotes the maximal torsion-free abelian quotient of a group $T$. Moreover, the natural action of $Sp_g (\Z)$ on
the target of $\tau_1$ extends to a rational, irreducible representation of the linear algebraic group $Sp_g(\C)$ on the complexification.

In his landmark paper \cite{J2}, Johnson showed that $\ker (\tau_1)=J^2(\pi_g)$ coincides with $K_g$. The Johnson kernel $K_g$ 
has since proved quite important in many studies. For example, it plays a key role in Morita's work on the Casson invariant \cite{Mo}.
Other equivalent definitions of the Johnson homomorphism $\tau_1$ for $\pi_g$, involving the cohomology rings of mapping tori, and 
the period map from Teichm\" uller theory, were proposed by Johnson in \cite{J4} (see also Hain \cite{H-msri} for complete proofs).
They led to alternative descriptions of $K_g$, and opened new perspectives in approaching mapping class groups and related moduli spaces.
Despite this, the basic finiteness questions on $K_g$ remained unanswered. 

\subsection{Strategy of proof} \label{ssi2}

Guided by the above results on $T_g$, we are going to define the Johnson kernel of a group $T$, denoted $K_T$, to be the kernel of 
the canonical projection, $T\surj T_{\abf}$. Assuming from now on $g\ge 3$, we recall that $K_{T_g}=K_g$. Since $T_g$ is finitely 
generated, work of Dwyer and Fried \cite{DF} (as refined in \cite{PS10}) on finiteness properties of Betti numbers in abelian covers implies that
$H_1(K_g, \Q)$ is finite dimensional if and only if $\V (T_g)$ is finite. Here, $\V (T_g)$ is a certain algebraic subvariety of the affine,
connected torus $\T^0 (T_g)=  \Hom ((T_g)_{\abf}, \C^*)$, called the {\em restricted characteristic variety} of $T_g$.

Our first step involves the geometry and the symmetry of $\V (T_g)$. The arithmetic group $Sp_g(\Z)$ naturally acts on $\T^0 (T_g)$
and preserves $\V (T_g)$. We use a powerful result in number theory due to Laurent \cite{Lau}, to show that either  $\V (T_g)$
is finite, or $\V (T_g)= \T^0 (T_g)$. This is reminiscent of a foundational result in algebraic geometry due to Arapura, who determined
in \cite{A} the qualitative structure of characteristic varieties associated to fundamental groups of quasi--Kahler manifolds; this makes the
Kahler problem for $T_g$ even more inciting.

Another important result on the geometry of characteristic varieties was obtained by Libgober \cite{Lib} (and refined in 
Lemma \ref{lem:tgcone}). It implies that the tangent cone at the unit $1\in \T^0 (T_g)$ of $\V (T_g)$ is contained in the so--called
{\em resonance variety} $\R (T_g)$. The resonance variety is an algebraic subvariety sitting inside the Lie algebra of the algebraic group
$\T^0 (T_g)$, invariant under the $Sp_g(\Z)$--action. 

Our second step is the analysis of  the geometry and symmetry of $\R (T_g)$. It is known that resonance varieties are closely connected to
associated graded Lie algebras. A deep result of Hain \cite{H} gives a presentation for the Lie algebra $\gr_{\bullet} (T_g)\otimes \C$. 
We combine in Theorem \ref{thm:rest} Hain's work with classical methods from the representation theory of $Sp_g(\C)$, to arrive at
the following conclusion:  $\R (T_3) = \Lie (\T^0 (T_3)) \ne \{0\}$, and $\R (T_g)= \{ 0\}$, for $g\ge 4$.

When $g\ge 4$, this implies that $\V (T_g) \ne \T^0 (T_g)$. Hence, $\V (T_g)$ must be finite, and $H_1(K_g, \Q)$ must be
finite dimensional, which proves Theorem \ref{thm:johnkf}. 

In Theorem \ref{thm:kintro}, we follow a similar approach. The {\em BNSR invariant}  of the
finitely generated group $T_g$, denoted $\Sigma (T_g)$, is a subset of $\Hom ((T_g)_{\abf}, \RR) \setminus \{0\}$. A
powerful result in geometric group theory, due to Bieri, Neumann and Strebel \cite{BNS}, implies that $K_g$ is
finitely generated if and only if $\Sigma (T_g)= \Hom ((T_g)_{\abf}, \RR) \setminus \{0\}$. Again, $Sp_g(\Z)$ 
naturally acts on $\Hom ((T_g)_{\abf}, \RR)$, preserving $\Sigma (T_g)$. 

When $T=\pi_1(M)$ is the fundamental group of a compact Kahler manifold, a key theorem of Delzant \cite{D}
describes $\Sigma (T)$ in terms of the geometry of holomorphic pencils on $M$. We deduce Theorem \ref{thm:kintro}
from the vanishing of $\R (T_g)$, by using Delzant's result, together with the $Sp_g(\Z)$--symmetry of $\Sigma (T_g)$.

The techniques developed in this paper lead to a surprising answer to another conjecture,
concerning the subgroups of the Johnson filtration associated to automorphism groups of
free groups. See \cite{PS10}.

\section{Main general results} \label{sec:abs}

As explained in the introduction, we adopt the group--theoretic viewpoint on $\Gamma_g$, $T_g$ and $K_g$.
This approach leads us to various other general results, that seem to be of independent interest. 
In this section, we discuss this more general viewpoint.

The natural $\aut (T)$-symmetry factors through $\out (T)$, for a good number of invariants of
a group $T$. The simplest examples are the abelianization, $T_{\ab}$, and the quotient of $T_{\ab}$
by its torsion, $T_{\abf}$. More well-known examples are provided by the (complex) cohomology ring
$H^{\bullet} T:= H^{\bullet}(T, \C)$ and the (complexified) graded Lie algebra associated to the lower central series,
$\cG_{\bullet}(T):= \gr_{\bullet}T \otimes \C$.

Our starting point  is to consider a group epimorphism, $p:\G \surj D$, 
with finitely generated kernel $T$. This is motivated by the defining exact sequence \eqref{eq:deftintro} of Torelli groups.
We examine three other known types of invariants with
natural outer symmetry, through the prism of the $D$-symmetry induced by the canonical homomorphism,
$D\to \out (T)$. Firstly, we look at the resonance varieties $\R^i_k(T)$ 
(i.e., the jump loci for a certain kind of
homology, associated to the ring $H^{\bullet} T$), sitting inside $H^1T$; they are reviewed in 
Section \ref{sec:res}, and their outer symmetry is discussed in Remark \ref{rk:dhres}.
For our purposes here, the most important resonance variety will be $\R (T):= \R^1_1(T)$, which is 
Zariski closed in $H^1 T$.  Secondly,
we inspect the characteristic varieties $\V^i_k(T)$ 
(i.e., the jump loci for homology with rank one complex local systems), lying inside the character torus
$\T (T):= \Hom (T_{\ab}, \C^*)$, and their intersection with the connected component of $1\in \T (T)$ of the
affine group $\T (T)$, $\T^0 (T):= \Hom (T_{\abf}, \C^*)$;
their definition is recalled in Section \ref{sec:dsim}, and their outer symmetry is explained in 
Lemma \ref{prop:dsim}. The restricted characteristic variety $\V (T):= \V^1_1(T) \cap \T^0 (T)$ is
Zariski closed in $\T^0 (T)$. 
Finally, we recollect in Section \ref{sec:rsigma} a couple of relevant facts about the
Bieri-Neumann-Strebel-Renz (BNSR) invariants, $\Sigma^q(T, \Z)\subseteq H^1(T, \RR)\setminus \{ 0\}$,
and we point out their outer symmetry in Lemma \ref{prop:rsim}. We will be particularly interested in 
$\Sigma (T):= \Sigma^1 (T, \Z)$. In particular, the resonance and characteristic varieties of $T_g$, 
as well as its BNSR--invariants, 
acquire a natural $Sp_g(\Z)$-symmetry, for $g\ge 3$. 

Our choice for the types of invariants was dictated by the fact that each of them controls a certain kind of
finiteness properties. We begin with resonance, for which this relationship seems to be new.
We devote Section \ref{sec:resfin} to the analysis of the connexions between triviality of resonance and
finiteness of various (completed) Alexander-type invariants, for finitely generated groups.

Let $K\subseteq T$ be a subgroup containing the derived group $T'$. Then $H_1K:= H_1(K, \C)$
becomes in a natural way a module over the group ring $\C T_{\ab}$, with module structure induced
by $T$-conjugation.
Its $I$-adic completion is denoted
$\widehat{H_1K}$, where $I \subseteq \C T_{\ab}$ is the augmentation ideal. When $K=T'$, 
$H_1K$ is the classical Alexander invariant from link theory (over $\C$). 

The technique of
$I$-adic completion was promoted in low-dimensional topology by Massey \cite{Mas}. A key result related to  
$I$-adic completion was obtained by Hain in \cite{H}, where he shows that $T_g$ is a $1$-formal group in the sense
of Sullivan \cite{Su}, for $g\ge 6$. Here is our first main abstract result. By Theorem \ref{thm:rest}, the Part \eqref{ai2} below
applies to $T=T_g$, for $g\ge 4$. 

\begin{thm} 
\label{thm:aintro}
Let $T$ be a finitely generated group.
\begin{enumerate}
\item \label{ai1}
Assume $T$ is $1$-formal. Then $\R^1_1(T) \subseteq \{ 0\}$ if and only if 
$\dim_{\C} \widehat{H_1T'}< \infty$.
\item \label{ai2}
For any subgroup $K\subseteq T$ containing $T'$, 
$\dim_{\C} \widehat{H_1K}< \infty$, if $\R^1_1(T) \subseteq \{ 0\}$.
\end{enumerate}
\end{thm}

Aiming at finer finiteness properties, we go on by examining characteristic varieties, in 
Section \ref{sec:dsim}. Here, we start from a basic result of Dwyer and Fried \cite{DF}, as refined
in \cite{PS10}. It says that the finiteness of Betti numbers, up to degree $q$, of normal subgroups with 
abelian quotient is detected precisely by the finiteness of the intersection between
characteristic varieties of type $\V^i_1$, for $i\le q$, and the corresponding subtorus of $\T$.

Given the $D$-symmetry of characteristic varieties, we are thus led to consider the following
context, inspired by Torelli groups. 
Let $L$ be a $D$-module which is finitely generated and free as an abelian group. Assume that $D$ is an
arithmetic subgroup of a simple $\C$-linear algebraic group $S$ defined over $\Q$, with
$\Q-\rank (S)\ge 1$. Suppose also that the $D$-action on $L$ extends to an irreducible, 
rational $S$-representation in $L\otimes \C$. (Note that the above assumptions are satisfied for $g\ge 3$
by $D=Sp_g(\Z) \subseteq Sp_g(\C)=S$ and $L=(T_g)_{\abf}$, due to Johnson's pioneering results on the
symplectic symmetry of Torelli groups from \cite{J0, J3}.)
The $D$-representation in $L$ gives rise to a natural $D$-action on the connected affine torus
$\T (L):= \Hom (L, \C^*)$. 

\begin{thm}
\label{thm:bintro}
If the $D$-module $L$ satisfies the above assumptions, then $\T (L)$ is geometrically $D$-irreducible,
that is, the only $D$-invariant, Zariski closed subsets of $\T (L)$ are either equal to $\T (L)$,
or finite.
\end{thm}

We deduce Theorem \ref{thm:bintro} from a deep result in diophantine geometry, due to
M. Laurent \cite{Lau}, in Section \ref{sec:dsim}. Note that the conclusion of our theorem above
is in marked contrast with the behavior of the induced $D$-representation in the affine space
$L\otimes \C$, for which $S$-invariant, infinite and proper Zariski closed subsets may well exist.
Theorem \ref{thm:bintro} enables us to obtain in Section \ref{sec:dsim} the following consequences
of the triviality of resonance. In the particular context of mapping class groups, we recover from Part \eqref{it}
below Theorem \ref{thm:johnkf}, in slightly stronger form. 

\begin{thm}
\label{thm:cintro}
Let $p:\G \surj D$ be a group epimorphism with finitely generated kernel $T$, having the property that
$\R^1_1(T)\subseteq \{ 0\}$. Assume $D\subseteq S$ is arithmetic, where the
$\C$-linear algebraic group $S$ is defined over $\Q$, simple, with
$\Q-\rank (S)\ge 1$. Suppose moreover that the canonical $D$-representation in $T_{\abf}$ 
extends to an irreducible, 
rational $S$-representation in $T_{\abf}\otimes \C$. The following hold.
\begin{enumerate}
\item \label{vt}
The (restricted) characteristic varieties $\V^1_k (T) \cap \T^0 (T)$ are finite, for $k\ge 1$.
\item \label{at}
If moreover $b_1(T)>1$, the Alexander polynomial $\Delta^T$ is a non-zero constant,
modulo the units of the group ring $\Z T_{\abf}$.
\item \label{it}
For any subgroup $N\subseteq T$, containing the kernel of the canonical map,
$T\surj T_{\abf}$, the first Betti number $b_1(N)$ is finite.
\end{enumerate}
\end{thm}

Note that the computation of characteristic varieties and Alexander polynomials can be 
a very difficult task, in general. What makes life easier in Theorem \ref{thm:cintro},
Parts \eqref{vt} and \eqref{at}, is the arithmetic symmetry. These two results hold in particular for $T=T_g$,
when $g\ge 4$. As we pointed out earlier, in this case we may infer from Theorem \ref{thm:aintro}\eqref{ai2}
that $\dim_{\C} \widehat{H_1 T_g'}<\infty$. 
Note that the finite-dimensionality of the (uncompleted) Alexander invariant $H_1(T_g', \C)$
cannot be deduced from Theorem \ref{thm:cintro}\eqref{vt}, since $(T_g)_{\ab}$ contains non-trivial
$2$-torsion, according to Johnson \cite{J3}. To the best of our knowledge, the finiteness of
$b_1(T_g')$ is an open question.

Theorem \ref{thm:cintro} may also be used in the particular case when $T=J^1(\FF_n)$, with $n\ge 4$,
and leads to the conclusion that the first Betti number of $J^2(\FF_n)$ is finite; see \cite{PS10}.

We investigate in Section \ref{sec:rsigma} the BNSR invariants, 
which control geometric finiteness properties of 
normal subgroups with abelian quotient \cite{BNS, BR}. In the context from \eqref{eq:deftintro}, we use 
their arithmetic symmetry to prove Theorem \ref{thm:kintro}.

\section{Associated graded Lie algebra} \label{sec:csim}

The symplectic symmetry is well-known to be an important tool for the study of Torelli
groups. In this section, we recall a basic result of Hain \cite{H}, related
to the symplectic symmetry at the level of the associated graded Lie algebra. 
Hain's starting point is  Johnson's pioneering work, which we review first.

Let $\Sigma_g$ be a closed, oriented, genus $g$ surface.  In the sequel we will assume that $g\ge 3$. 
In this range, the group $T_g$ is finitely generated. For a group $T$, we denote by
$T_{\abf}$ the quotient of its abelianization $T_{\ab}$ by the torsion subgroup. Among other things,
Johnson \cite{J0, J3} gave a very convenient description of $(T_g)_{\abf}$, in the following way.
Fix a symplectic basis of $H:= H_1(\Sigma_g, \Z)$, $\{ a_1,\dots, a_g,b_1,\dots, b_g \}$, and denote by 
$\omega= \sum_{i=1}^g a_i \wedge b_i \in \bigwedge^2 H$ the symplectic form. 
Let $Sp_g(\Z)$ be the group of symplectic automorphisms of $H$. Note that the
$Sp_g(\Z)$-action on $H$ canonically extends to a $Sp_g(\Z)$-action on the exterior algebra
$\bigwedge^* H$, by graded algebra automorphisms. Consider the $Sp_g(\Z)$-equivariant
embedding, $H\hookrightarrow \bigwedge^3 H$, given by $h\in H\mapsto h\wedge \omega \in \bigwedge^3 H$, 
and denote by $L$ the $Sp_g(\Z)$-module $\bigwedge^3 H/H$. Johnson's homomorphism
constructed in \cite{J0},  $\tau_1 \colon T_g \to L$, has the following properties.

\begin{theorem}[Johnson]
\label{thm:john}
The group homomorphism $\tau_1$ is $\G_g$-equivariant, with respect to the (left) conjugation action
on $T_g$ induced by \eqref{eq:deftintro}, and the restriction of the $Sp_g(\Z)$-action on $L$ via $p$.
It induces a $Sp_g(\Z)$-equivariant isomorphism, $\tau_1 \colon (T_g)_{\abf} \isom L$.
\end{theorem}

Note that the arithmetic group $Sp_g(\Z)$
is a Zariski dense subgroup of the semisimple algebraic group $Sp_g(\C)$; see e.g. \cite[Corollary 5.16]{R2}. 
Setting $H(\C)=H\otimes \C$ and $L(\C)=L\otimes \C$, note also that the canonical representation
of $Sp_g(\Z)$, coming from \eqref{eq:deftintro}, in $(T_g)_{\ab} \otimes \C \simeq L(\C)$, extends to
a rational representation of $Sp_g(\C)$. This symplectic symmetry propagates to higher degrees,
in the following sense.

Recall that the {\em associated graded Lie algebra} (with respect to the lower central series) of 
a group $T$, $\gr_{\bullet} T$, is generated as a Lie algebra by  $\gr_1 T= T_{\ab}$, since 
the Lie bracket is induced by the group commutator. 

\begin{lemma}
\label{lem:ext}
Given a group extension,
\begin{equation}
\label{eq:tabs}
1\to T\rightarrow \G \rightarrow D\to 1\, ,
\end{equation}
assume that $T$ is finitely generated, $D$ is a Zariski dense subgroup of a complex 
linear algebraic group $S$, and the $D$-action on $T_{\ab}$ extends to a rational representation of
$S$ in $T_{\ab}\otimes \C$. Then the $D$-action on $\gr_{\bullet} T$ extends to an action of $S$ on
$\gr_{\bullet} T \otimes \C$, in the category of graded rational representations;
moreover, every $s\in S$ acts on $\gr_{\bullet} T\otimes \C$ by a graded Lie algebra automorphism.
\end{lemma}

\begin{proof}
Presumably this result is well-known to the experts. Being unable to find a reference, we decided to
include a proof. Set $\TT_{\bullet} := \gr_{\bullet} T \otimes \C$, noting that $D$ acts on $\TT_{\bullet}$ by graded Lie algebra
automorphisms. For each $q\ge 1$, denote by $K_q \subseteq \TT_1^{\otimes q}$ the kernel of the
linear surjection sending $t_1\otimes \cdots \otimes t_q$ to 
$\ad_{t_1}\circ \cdots \ad_{t_{q-1}}(t_q) \in \TT_q$. By construction, $\TT_q = \TT_1^{\otimes q}/K_q$.
Since $D\subseteq S$ is Zariski dense, the linear
subspace $K_q$ is $S$-invariant. It remains to show that the rational representations of $S$ in $\TT_{\bullet}$
constructed in this way, which extend the $D$-action coming from \eqref{eq:tabs}, have the property that
$s[a,b]=[sa, sb]$, for $s\in S$, $a\in \TT_q$ and $b\in \TT_r$. This in turn is easily proved by
induction on $q$. Induction starts with $q=1$, by noting that the iterated Lie bracket,
$\TT_1^{\otimes r+1}\surj \TT_{r+1}$, is $S$-equivariant by construction. For the inductive step,
write $a=[t,a']$, with $t\in \TT_1$ and $a'\in \TT_{q-1}$, and use the Jacobi identity to conclude.
\end{proof}

By Lemma \ref{lem:ext} and Theorem \ref{thm:john}, we have a short exact sequence of rational 
$Sp_g(\C)$-representations, 
\begin{equation}
\label{eq:grlow}
0\to \mathcal{K}\to \bigwedge^2 \gr_1T_g \otimes \C \stackrel{\beta}{\rightarrow} \gr_2T_g \otimes \C \to 0\, ,
\end{equation}
where $\beta$ denotes the Lie bracket and $\mathcal{K}$ is by definition $\ker (\beta)$. Hain \cite{H} 
computed the associated exact sequence of
$\sp_g$-modules, where $\sp_g$ is the Lie algebra of $Sp_g(\C)$. To describe his result, we follow the
conventions from \cite[Section 6]{H}. Our references for algebraic groups (respectively Lie algebras) are
\cite{HG} (respectively \cite{HL}).  

The Lie algebra of the maximal diagonal torus in $Sp_g(\C)$ is denoted $\h$, and has coordinates
$t=(t_1,\dots, t_g)$. Let $\Phi\subseteq \h^*$ be the corresponding root system, with the standard choice
of positive roots, $\Phi^+ = \{ t_i-t_j, t_i+t_j \mid 1\le i<j \le g \} \cup \{ 2t_i \mid 1\le i\le g \}$.
Let $\sp_g := \s= \n^- \oplus \h \oplus \n^+$ be the canonical decomposition of the Lie algebra. We denote
by $B$ the associated Borel subgroup of $Sp_g(\C):= S$, with unipotent radical $U$; the Lie algebra of
$B$ is $\h \oplus \n^+$, and the Lie algebra of $U$ is $\n^+$. We work with the finite-dimensional
$\s$-modules associated to rational representations of $S$. The irreducible ones are of the form $V(\lambda)$,
where the dominant weight $\lambda$ is a positive integral linear combination of the fundamental weights,
$\{ \lambda_j(t)= t_1+\cdots +t_j \mid 1\le j\le g \}$. 

It follows from Theorem \ref{thm:john} that $\gr_1 T_g \otimes \C =V(\lambda_3)$, as
$\sp_g$-modules. According to \cite[Lemma 10.2]{H}, all irreducible submodules of 
$\bigwedge^2 V(\lambda_3)$ occur with multiplicity one, and $\bigwedge^2 V(\lambda_3)$ contains
$V(2\lambda_2) \oplus V(0)$ as a submodule.

\begin{theorem}[Hain]
\label{thm:hainhol}
The $\sp_g$-map $\beta$ from \eqref{eq:grlow} is the canonical $\sp_g$-equivariant projection of
$\bigwedge^2 V(\lambda_3)$ onto the submodule $V(2\lambda_2) \oplus V(0)$.
\end{theorem}

\section{Resonance varieties of Torelli groups} \label{sec:res}

In this section, we use representation theory to compute the resonance varieties (in degree $1$ and
for depth $1$) of the Torelli groups $T_g$, for $g\ge 3$, over $\C$.

We begin by reviewing the {\em resonance varieties}, $\R^i_d (A^{\bullet})$, associated to a connected, 
graded-commutative $\C$-algebra $A^{\bullet}$, for (degree) $i\ge 0$ and (depth) $d\ge 1$. Given
$a\in A^1$, denote by $\mu_a$ left-multiplication by $a$ in $A^{\bullet}$, noting that $\mu_a^2 =0$,
due to graded-commutativity. Set
\begin{equation}
\label{eq:defres}
\R^i_d (A^{\bullet})= \{ a\in A^1 \mid \dim_{\C} H^i(A^{\bullet}, \mu_a) \ge d \}\, .
\end{equation}

\begin{remark}
\label{rk:dhres}
It seems worth pointing out that, given an arbitrary group extension \eqref{eq:tabs}, one has the
following symmetry property, related to resonance. By standard homological algebra (see e.g. \cite{B}),
the conjugation action of $\G$ on $T$ induces a (right) action of $D$ on $A^{\bullet}= H^{\bullet}(T, \C)$,
by graded algebra automorphisms. We infer from \eqref{eq:defres} that the $D$-representation on
$A^1$ leaves $\R^i_d (A^{\bullet})$ invariant, for all $i$ and $d$.
\end{remark}

In this paper, we need to consider only the resonance varieties $\R_d (A^{\bullet}):= \R^1_d(A^{\bullet})$, 
which plainly depend only on the
co-restriction of the multiplication map, $\cup :\wedge^2 A^1 \to A^2$, to its image. When
$\dim_{\C}A^1 <\infty$, it is easy to see that each $\R_d (A^{\bullet})$ is a Zariski closed, 
homogeneous subvariety
of the affine space $A^1$. For a connected space $M$, we denote $\R^i_d (H^{\bullet}(M, \C))$ by 
$\R^i_d (M)$. It is equally easy to check that $\R^1_d (M)= \R^1_d (M^{(2)})$, for all $d$,
where $M^{(q)}$ denotes the $q$--skeleton of a connected CW-complex $M$.
For a group $G$, $\R^i_d (G)$ means $\R^i_d (K(G, 1))$. 
By considering a classifying map, it is
immediate to check that $\R_d(M)=\R_d(\pi_1(M))$. We will abbreviate $\R_1$ by $\R$.

There is a well-known, useful relation between cup-product in low degrees and group commutator, see
Sullivan \cite{S}, and also Lambe \cite{L} for details. For a group $G$ with finite first Betti 
number, there is a short exact sequence,
\begin{equation}
\label{eq:grcup}
0\to (\gr_2 G\otimes \C)^* \stackrel{d}{\rightarrow} H^1G\wedge H^1G \stackrel{\cup}{\rightarrow}
H^2G \, ,
\end{equation}
where cohomology is taken with $\C$-coefficients and $d$ is dual to the Lie bracket $\beta$ from
\eqref{eq:grlow}. (Note that the only finiteness assumption needed in the proof from \cite{L}
is $b_1(G)<\infty$.)

In what follows, we retain the notation from Section  \ref{sec:csim}. Set $V=L(\C)^*$ and 
$\R=\R(T_g)$. We infer from \eqref{eq:grlow} and \eqref{eq:grcup} that $\R\subseteq V$ is
$Sp_g(\C)$-invariant. Here is our first step in exploiting the complex symmetry from 
Theorem \ref{thm:hainhol}.

\begin{lemma}
\label{lem:max}
If $\R\ne \{0\}$, then $\R$ contains a maximal vector of the $\sp_g$-module $V$. 
\end{lemma}

\begin{proof}
This is an easy consequence of Borel's fixed point theorem \cite[Theorem 21.2]{HG}, which
guarantees the existence of a $B$-invariant line, $\C\cdot v \subseteq \R$. Invariance under the action 
of the maximal torus implies that $v$ belongs to a weight space of the $\h$-action on $V$, that is,
$v\in V_{\lambda}$ for some $\lambda \in \h^*$. Finally, $\n^+\cdot v=0$ follows from $U$-invariance.
\end{proof}

Since $-1$ belongs to the Weyl group of $\sp_g$, all finite-dimensional representations of $\sp_g$ are
self-dual \cite[Exercise 6 on p.116]{HL}. This remark leads to the following dual reformulation of
Theorem \ref{thm:hainhol}, via \eqref{eq:grcup}. 

\begin{lemma}
\label{lem:dual}
Set $V=V(\lambda_3)$. Then $H^1T_g=V$, as $\sp_g$-modules, and the kernel of the cup-product,
$\cup: \bigwedge^2 H^1T_g \to H^2T_g$, is  $V(2\lambda_2) \oplus V(0)$.
\end{lemma}

Our main result in this section is the following.

\begin{theorem}
\label{thm:rest}
For $g=3$, $\R(T_3)=H^1T_3$, while $\R(T_g)=\{0\}$, for $g\ge 4$. 
\end{theorem}

The assertion for $g=3$ is an immediate consequence of Lemma \ref{lem:dual}, since in this case
$\bigwedge^2V= V(2\lambda_2) \oplus V(0)$, cf. Lemma 10.2 from \cite{H}. So, we will assume 
in the sequel that $g\ge 4$ and $\R(T_g)\ne \{0\}$, and use Lemma \ref{lem:max} to derive a contradiction.

Set $V(0)=\C \cdot z_0$. We will need explicit maximal vectors, $v_0$ for $V(\lambda_3)$ and
$u_0$ for $V(2\lambda_2)$. To this end, we recall that $V(\lambda_3)= (\bigwedge^3 H/H)\otimes \C$,
where the $\sp_g$-action on $\bigwedge^{\bullet} H$ is by algebra derivations; in particular,
$s\cdot x\wedge y\wedge z= s\cdot x\wedge y \wedge z+ x\wedge s\cdot y\wedge z+ x\wedge y\wedge s\cdot z$,
for $s\in \s$ and $x,y,z\in H$. Set $v_0'= a_1\wedge a_2\wedge a_3$, and 
$u_0'= \sum_{k=3}^g ( a_1\wedge a_2\wedge a_k)\bigwedge ( a_1\wedge a_2\wedge b_k)$. Denote by $v_0$
the class of $v_0'$ in $V=V(\lambda_3)$, and let $u_0$ be the class of $u_0'$ in $\bigwedge^2V$.
Using the explicit description of $\h \oplus \n^+$ from \cite[Section 6]{H}, it is straightforward 
to check that $v_0$ has weight $\lambda_3$, $u_0$ has weight $2\lambda_2$, and 
$\n^+\cdot v_0=\n^+\cdot u_0=0$.

To verify that both $v_0$ and $u_0$ are non-zero, we will resort to the $Sp_g(\Z)$-equivariant contraction 
constructed by Johnson in \cite[p.235]{J0}, $C: \bigwedge^3 H \to H$, given by
\begin{equation}
\label{eq:contr}
C(x\wedge y\wedge z)= (x\cdot y)z+ (y\cdot z)x+ (z\cdot x)y \, ,
\end{equation}
where the dot designates the intersection form on $H$. Set $L'=\ker(C)$, and $L'(\C)=L' \otimes \C$.
It follows from \cite[pp.238-239]{J0} that there is an induced $\sp_g$-isomorphism, 
\begin{equation}
\label{eq:lprim}
L'(\C)\isom L(\C)\, .
\end{equation}
Clearly, $0\ne v_0'\in L'$ and $0\ne u_0'\in \bigwedge^2 L'$. Hence, both $v_0$ and $u_0$ are non-zero,
as needed.

We infer from uniqueness of maximal vectors \cite[Corollary 20.2]{HL} that necessarily $v_0\in \R=\R(T_g)$,
if $\R\ne \{0\}$; see Lemma \ref{lem:max}. Denote by $W$ the $\s$-submodule $V(2\lambda_2) \oplus V(0)$ of
$\bigwedge^2V$ from Lemma \ref{lem:dual}. By definition \eqref{eq:defres} and Lemma \ref{lem:dual}, 
$v_0\in \R$ if and only if $\im (\mu_0)\cap W \ne 0$, where $\mu_0: V\to \bigwedge^2V$ denotes 
left-multiplication by $v_0\in V$ in the exterior algebra.

Since $\n^+\cdot v_0=0$, we conclude that $\mu_0$ is $\n^+$-equivariant. Since $v_0\in V_{\lambda_3}$,
it follows that
\begin{equation}
\label{eq:wshift}
\mu_0 (V_{\lambda'})\subseteq (\bigwedge^2V)_{\lambda_3 +\lambda'}\, ,
\end{equation}
for each weight subspace $V_{\lambda'}\subseteq V$. 

\begin{lemma}
\label{lem:eng}
If $\R(T_g)\ne \{0\}$, then $\im (\mu_0)\cap W$ contains a non-zero vector killed by $\n^+$. 
\end{lemma}

\begin{proof}
Taking into account the preceding discussion, the assertion will follow from Engel's theorem
\cite[Theorem 3.3]{HL}, applied to  the Lie algebra $\n^+$ and the $\n^+$-module 
$\im (\mu_0)\cap W$. It is enough to check that the action of any $n\in \n^+$ on $W$ is
nilpotent. This in turn follows from the fact that 
$s_{\alpha_1}\cdots s_{\alpha_r}\cdot W_{\lambda'} \subseteq  W_{\lambda''}$, where 
$\lambda''=\lambda' + \sum_{i=1}^r \alpha_i$, for any $\alpha_1,\dots ,\alpha_r \in \Phi^+$ and
any non-trivial weight space $W_{\lambda'}$. Indeed, $ W_{\lambda''}=0$ for $r$ big enough.
To check the vanishing claim, one may invoke \cite[Theorem 20.2(b)]{HL}, and then use 
a height argument.
\end{proof}

\begin{lemma}
\label{lem:non}
The following hold.
\begin{enumerate}
\item \label{n1}
For $g\ge 4$, the vector $v_0\wedge u_0\in \bigwedge^3V$ is non-zero.
\item \label{n2}
The class of $( a_1\wedge a_2\wedge a_3)\bigwedge ( b_1\wedge b_2\wedge b_3)$ in
$\bigwedge^2V$ does not belong to $V(0)$.
\end{enumerate}
\end{lemma}

\begin{proof}
\eqref{n1} By \eqref{eq:lprim}, it is enough to show that $v_0'\wedge u_0'\in \bigwedge^3(\wedge^3 H)$
is non-zero. This is clear, since 
$v_0'\wedge u_0'=\sum_{k=4}^g ( a_1\wedge a_2\wedge a_3)\bigwedge ( a_1\wedge a_2\wedge a_k)
\bigwedge ( a_1\wedge a_2\wedge b_k)$.

\eqref{n2} Again by \eqref{eq:lprim}, it suffices to verify that the element 
$e=( a_1\wedge a_2\wedge a_3)\bigwedge ( b_1\wedge b_2\wedge b_3)\in \bigwedge^2(\wedge^3 H)\otimes \C$
is not killed by $\sp_g$, since $b_1\wedge b_2\wedge b_3 \in L'$. Indeed, 
$T_1\cdot e= ( a_1\wedge a_2\wedge a_3)\bigwedge ( a_1\wedge b_2\wedge b_3) \ne 0$, where
$T_1\in \sp_g$ is described in \cite[\S 6]{H}.
\end{proof}

We will finish the proof of Theorem \ref{thm:rest}, by showing that Lemmas \ref{lem:eng} 
and \ref{lem:non} lead to a contradiction, assuming $g\ge 4$ and $\R(T_g)\ne \{0\}$. 

Indeed, in this case there is $v\in V(\lambda_3)$ such that 
\begin{equation}
\label{eq:contra}
0\ne v_0\wedge v= w_0 + w\, ,
\end{equation}
with $w_0\in V(0)$, $w\in V(2\lambda_2)$ and $\n^+\cdot w=0$, by Lemma \ref{lem:eng}.

If $w\ne 0$, then $w\in \C^*\cdot u_0$, by uniqueness of maximal vectors. Taking the weight
decomposition of $v$, it follows from \eqref{eq:wshift} that we may suppose that $w_0=0$ in
\eqref{eq:contra}. Therefore, $v_0\wedge u_0=0\in \bigwedge^3V$, contradicting 
Lemma \ref{lem:non}\eqref{n1}.

If $w=0$, then $v_0\wedge v\in \C^*\cdot z_0$, for some $v\in V(\lambda_3)_{-\lambda_3}$, by
the same argument on weight decomposition as before. Since the weights $\lambda_3$ and $-\lambda_3$
are conjugate under the action of the Weyl group, the weight space $V(\lambda_3)_{-\lambda_3}$
is one-dimensional, generated by the class $\overline{v}_0$ of $b_1\wedge b_2\wedge b_3 \in L'$
in $V$; see \cite[Theorems 20.2(c) and 21.2]{HL}, \cite[Section 6]{H} and \eqref{eq:lprim}.
Therefore, $v_0\wedge \overline{v}_0\in V(0)$, contradicting Lemma \ref{lem:non}\eqref{n2}.
The proof of Theorem \ref{thm:rest} is thus completed.

\section{Resonance and finiteness properties} \label{sec:resfin}

We devote this section to a general discussion of finiteness properties related to resonance,
with an application to Torelli groups. Unless otherwise mentioned, we work with $\C$-coefficients.
For a graded object $\OO_{\bullet}$, the notation $\OO$ means that we forget the grading.

We need to review a couple of key notions. We begin with the {\em holonomy Lie algebra}
associated to a connected CW-complex $M$ with finite $1$-skeleton, $\cH_{\bullet}(M)$. This is a
quadratic graded Lie algebra, defined as the quotient of the free Lie algebra generated by $H_1M$,
$\bL_{\bullet}(H_1M)$, graded by bracket length, modulo the ideal generated by the image of
the comultiplication map, $\partial_M: H_2M\to \bigwedge^2 H_1M$; here we are using the standard
identification $\bigwedge^2 H_1M \equiv \bL_{2}(H_1M)$, given by the Lie bracket. Note that the dual 
of $\partial_M$ is the cup-product, $\cup_M: \bigwedge^2 H^1M\to H^2M$. 

When $G$ is a finitely generated group and $M=K(G,1)$, $\cH_{\bullet}(M)$ 
is denoted $\cH_{\bullet}(G)$. By considering a classifying map, it is easy to see that
$\cH_{\bullet}(M)= \cH_{\bullet}(\pi_1(M))$, for a connected CW-complex $M$ with finite $1$-skeleton. 
The associated graded Lie algebra 
$\gr_{\bullet}G \otimes \C$, denoted $\cG_{\bullet}(G)$, is also finitely generated in degree one,
but not necessarily quadratic. The canonical identification, $H_1G \equiv \cG_1(G)$, extends to
a graded Lie algebra epimorphism, $\bL_{\bullet}(H_1G) \surj \cG_{\bullet}(G)$. By \eqref{eq:grcup},
this factors to a graded Lie algebra surjection,
\begin{equation}
\label{eq:grhol}
\cH_{\bullet}(G) \surj \cG_{\bullet}(G) \, .
\end{equation}

Let $P_{\bullet}$ be the polynomial algebra $\Sym (H_1G)$, endowed with the usual grading.
For a Lie algebra $\cH$, denote by $\cH':= [\cH, \cH]$ the derived Lie algebra, and by
$\cH'':= [\cH', \cH']$ the second derived Lie algebra.
The exact sequence of graded Lie algebras
\begin{equation}
\label{eq:binf}
0\to \cH'_{\bullet}(G)/\cH''_{\bullet}(G) \rightarrow \cH_{\bullet}(G)/\cH''_{\bullet}(G) 
\rightarrow  \cH_{\bullet}(G)/\cH'_{\bullet}(G)\to 0
\end{equation}
yields a positively graded $P_{\bullet}$-module structure on the {\em infinitesimal Alexander invariant},
$\bb_{\bullet}(G):= \cH'_{\bullet}(G)/\cH''_{\bullet}(G)$, induced by the adjoint action. 
A finite $P_{\bullet}$-presentation of
$\bb_{\bullet}(G)$ is described in \cite[Theorem 6.2]{PS}. 

Let $E_{k-1}(\bb(G))\subseteq P$ be the {\em Fitting ideal} generated by the codimension
$k-1$ minors of a finite $P$-presentation for $\bb (G)$. The definition of $E_{k-1}$ does not depend on
the $P$--presentation; see \cite[Corollary 20.4]{E}. 
Denote by $\W_k(G)\subseteq H^1G$ 
the zero set of $E_{k-1}(\bb(G))$. The next lemma explains the relationship between the
infinitesimal Alexander invariant and the resonance varieties in degree one.

\begin{lemma}
\label{lem:infares}
Let $G$ be a finitely generated group. Then the equality
\[
\W_k (G) \setminus \{0\}=\R_k(G)\setminus \{0\}
\]
holds for all $k \geq 1$.
\end{lemma}

\begin{proof}
We know from \cite[Theorem 6.2]{PS} that $\bb (G)=\coker (\nabla)$, as $P$-modules, where
\begin{equation}
\label{eq:bpres}
\nabla:= \delta _3 + \id \otimes \partial _G\colon 
P\otimes\Big(\bigwedge\nolimits^3 H_1G \oplus H_2G \Big) \to 
P\otimes \bigwedge\nolimits^2 H_1G \, ,
\end{equation}
and the $P$-linear map $\delta_3$ is given by
$ \delta _3 (a \wedge b  \wedge c)
= a\otimes b   \wedge c +b \otimes c  \wedge a +c 
\otimes a \wedge b$, for $a,b,c\in H_1G$. A polynomial $f\in P= \Sym (H_1G)$ may be evaluated at
$z\in H^1 G$. It will be convenient to denote by $F(z)$ the matrix over $\C$ obtained by evaluating 
at $z$ all entries of the $P$--matrix $F$. This may be done for $\nabla$ and all Koszul differentials,
$\delta_i \colon P\otimes \bigwedge^i H_1G \rightarrow P\otimes\bigwedge^{i-1} H_1G$. 

Pick any $z \in H^1G \setminus \{0\}$. By linear algebra, we infer that 
$z \in \W_k (G)$ if and only if $
\dim _{\C}\coker (\nabla (z)) \geq k$. Consider the exact cochain complex
$(\bigwedge ^{\bullet} H^1G, \lambda _z)$, where 
$\lambda _z$ denotes left multiplication by $z$, and the dual exact chain complex, 
$(\bigwedge^{\bullet} H_1G, {}^{\sharp} \lambda _z)$. It is straightforward 
to check that the restriction of $ {}^{\sharp} \lambda _z$ to 
$\bigwedge ^{i} H_1G$ equals $\delta _i (z)$, for all $i$. 

We obtain from exactness the following isomorphism:
\[
\coker (\nabla (z)) \cong \im ( \delta _2(z))/ 
\im ( \delta _2(z)\circ \partial _G)\, .
\]
By exactness again,  $ \dim _{\C}\im ( \delta _2(z))-1$, where $n:=b_1(G)$. Hence 
$z \in \W_k (G)$ if and only if 
$\rank ( \delta _2(z)\circ \partial _G) \leq n-1-k$. 

The linear map dual to the restriction of $ {}^{\sharp} \lambda _z$ to 
$\bigwedge ^{2} H_1G$, $\delta _2(z)$,  is $\lambda _z \colon H^1 G \to \bigwedge^2 H^1 G$. 
By construction, the dual of $\partial_G$ is the cup--product $\cup_G \colon \bigwedge^2 H^1 G \to H^2 G$.
Hence, the linear map dual to $\delta_2 (z) \circ \partial_G$ is the map $\mu_z \colon H^1 G \to H^2 G$
from the definition of resonance \eqref{eq:defres}. Consequently, $z \in \W_k (G)$ if and only if
$\rank ( \mu_z) \leq n-1-k$, that is, if and only if $z \in \R_k(G)$; see again definition \eqref{eq:defres}.
\end{proof}

We may spell out our first general result relating resonance and finiteness.

\begin{theorem} \label{thm:resinf}
Let $G$ be a finitely generated group. Then $\R_1^1(G) \subseteq \{ 0\}$ if and only if
$\dim_{\C} \bb(G)< \infty$.
\end{theorem}

\begin{proof}
Lemma \ref{lem:infares} yields in particular the equality
$\R_1^1(G)= \Zero (\ann \bb(G))$, away from the origin. 
This is a consequence of the fact that the ideals $E_{0}(\bb(G))$ and $\ann \bb(G)$ have the same radical;
see \cite[Proposition 20.7]{E}. Let $\m \subseteq P$ be the maximal ideal of $0\in H^1G$. 

If $\dim_{\C} \bb(G)< \infty$, then $\m^k \subseteq \ann \bb(G)$, for some $k$, by degree
inspection. Taking zero sets, we obtain that $\R_1^1(G) \subseteq \{ 0\}$.

Conversely, the assumption $\R_1^1(G) \subseteq \{ 0\}$ implies $\m \subseteq \sqrt{\ann \bb(G)}$,
by Hilbert's Nullstellensatz. Therefore, $\m^k \subseteq \ann \bb(G)$, for some $k$. Since
$\bb(G)$ is finitely generated over $P$, we infer that $\dim_{\C} \bb(G)< \infty$.
\end{proof}

Theorem \ref{thm:resinf} has several interesting consequences. To describe them, we first
recall a couple of notions from rational homotopy theory. A {\em Malcev Lie algebra} (in the 
sense of Quillen) is a Lie algebra endowed with a decreasing vector space filtration 
satisfying certain axioms; see \cite[Appendix A]{Q}, where Quillen associates in a functorial way
a Malcev Lie algebra, $\M(G)$, to an arbitrary group $G$. The construction goes as follows.
Start with the group ring $\C G$, and his standard Hopf algebra structure. The $I$-adic completion,
$\widehat{\C G}$, becomes a complete Hopf algebra. The Malcev Lie algebra $\M(G)$ consists of the
primitive elements of $\widehat{\C G}$, with Lie bracket given by the algebra commutator in  $\widehat{\C G}$,
and filtration induced by the completion filtration of $\widehat{\C G}$.

Following Sullivan \cite{Su},
we say that a finitely generated group $G$ is {\em $1$-formal} if $\M(G)$ is the completion
with respect to degree of a quadratic Lie algebra, as a filtered Lie algebra. In  Theorem 1.1 from
\cite{H}, Hain proves that the Torelli group $T_g$ is $1$-formal, for $g\ge 6$ (but $T_3$ is not $1$-formal). 
Fundamental groups of compact Kahler manifolds are $1$-formal, as shown by
Deligne, Griffiths, Morgan and Sullivan in \cite{DGMS}. Many other
interesting examples of $1$-formal groups are known; see e.g. \cite{DPS} and the references therein.

Next, we review the {\em Alexander invariant}, $B(G)= G'_{\ab} \otimes \C$, associated to a 
finitely generated group $G$. The exact sequence of groups
\begin{equation}
\label{eq:bgr}
1\to G'/G'' \to G/G'' \to G/G'\to 1
\end{equation}
may be used to put a finitely generated module structure on $B(G)$, induced by conjugation, over
the Noetherian group ring $\C G_{\ab}$. We denote by $I\subseteq \C G_{\ab}$ the augmentation ideal,
and by $\widehat{B}$ the $I$-adic completion of a $\C G_{\ab}$-module $B$. 

Given a graded vector space $\V_{\bullet}$, $\widehat{\V}_{\bullet}$ means completion 
with respect to the degree filtration: $\widehat{\V}_{\bullet}= \varprojlim_{q}\V/F_q$,
where $F_q= \V_{\ge q}$. The canonical filtration of $\widehat{\V}_{\bullet}$ is
$\widehat{F}_q= \ker(\pi_q)$, where $\pi_q: \widehat{\V}_{\bullet}\to \V/F_q$ is the $q$-th
projection of the inverse limit. The next corollary proves Theorem \ref{thm:aintro}\eqref{ai1}.

\begin{corollary}
\label{cor:resb}
Let $G$ be a finitely generated, $1$-formal group. Then $\R_1^1(G) \subseteq \{ 0\}$ if and only if
$\dim_{\C} \widehat{B(G)}< \infty$. In particular, $\R_1^1(G) \subseteq \{ 0\}$, when
$\dim_{\C} B(G)< \infty$.
\end{corollary}

\begin{proof}
Let $\m \subseteq P$ be the maximal ideal of the origin. The $1$-formality of $G$ provides 
a vector space isomorphism between $\widehat{B(G)}$ and the $\m$-adic completion of $\bb(G)$,
cf. Theorem 5.6 from \cite{DPS}. Since $\bb_{\bullet}(G)$ is generated in degree $0$, its 
$\m$-adic completion coincides with the degree completion, $\widehat{\bb_{\bullet}(G)}$.
Clearly, $\bb_{\bullet}(G)$ and $\widehat{\bb_{\bullet}(G)}$ are simultaneously finite-dimensional.
Hence, our first claim follows from Theorem \ref{thm:resinf}. For the second claim, simply note 
that the completion of a finite-dimensional vector space is again finite-dimensional.
\end{proof}

\begin{remark}
\label{rk:anca}
The finitely generated group $G$ discussed in Example 6.4 from \cite{PS-bns} satisfies 
$\R_1^1(G) \subseteq \{ 0\}$, yet $\dim_{\C}B(G)=\infty$. It can be shown that $G$ is
$1$-formal. Consequently, $\R_1^1(G) \subseteq \{ 0\}$ is only a necessary condition for the
finite-dimensionality of $B(G)$, in general.

In the particular case when $G$ is nilpotent, the condition $\dim_{\C} B(G)< \infty$ is
automatically satisfied, since $G'$ is finitely generated. We recover in this way from 
Corollary \ref{cor:resb} the resonance obstruction to $1$-formality of finitely generated
nilpotent groups found by Carlson and Toledo in \cite[Lemma 2.4]{CT}. We refer the reader to
Macinic \cite{M}, for similar higher-degree obstructions to the formality of a finitely
generated nilpotent group.

As we shall see below, the main point in Corollary \ref{cor:resb} is the fact that the finite-dimensionality of
$\widehat{B(G)}$ forces $\R_1^1(G) \subseteq \{ 0\}$, when $G$ is $1$-formal. We point out that 
$1$-formality is needed for this implication. Indeed, let $G$ be the finitely generated nilpotent Heisenberg group with
$b_1(G)=2$. As noted before, $\dim_{\C} \widehat{B(G)}< \infty$. Yet, the resonance variety  $\R_1^1(G)= \C^2$
is non-trivial; see for instance \cite[Proposition 5.5]{M}.
\end{remark}

Let $G$ be an arbitrary finitely generated group, and $K\subseteq G$ a subgroup
containing $G'$. Clearly, $K$ is normal in $G$, and $G$-conjugation makes $H_1 K$ a finitely generated module
over the Noetherian group ring $\C [G/K]$. By restriction via the canonical epimorphism, $G/G' \surj G/K$, 
$H_1 K$ becomes a finitely generated $\C G_{\ab}$-module. We may now state
our next main result from this section, which proves Theorem \ref{thm:aintro}\eqref{ai2}.

\begin{theorem}
\label{thm:findex}
Let $G$ be a finitely generated group, and $K\subseteq G$ a subgroup
containing $G'$. If $\R_1^1(G) \subseteq \{ 0\}$,
the vector space $\widehat{H_1K}$ is finite-dimensional.
\end{theorem}

\begin{proof}
We  first treat the particular case $K=G'$, where we know from Theorem \ref{thm:resinf} that
the vector space $ \cH_{\bullet}(G)/\cH''_{\bullet}(G)$ is finite-dimensional. The canonical graded
Lie algebra surjection, $\cG_{\bullet}(G) \surj \cG_{\bullet}(G/G'')$, composed with the epimorphism
\eqref{eq:grhol}, gives a graded Lie algebra surjection, $\cH_{\bullet}(G) \surj \cG_{\bullet}(G/G'')$,
that factors through an epimorphism
\begin{equation}
\label{eq:hepi}
\cH_{\bullet}(G)/ \cH''_{\bullet}(G) \surj \cG_{\bullet}(G/G'') \, .
\end{equation}
It follows from \eqref{eq:hepi} that $\cG_{\bullet}(G/G'')=0$, for $\bullet >>0$.

On the other hand, we have an isomorphism,
\begin{equation}
\label{eq:mass}
I^q \cdot H_1G' /I^{q+1} \cdot H_1G' \simeq \cG_{q+2}(G/G'') \, ,
\end{equation}
for $q\ge 0$; see \cite[pp.400--401]{Mas}. We infer from \eqref{eq:mass} that the $I$-adic filtration of $H_1G'$ 
stabilizes, for $q>>0$. Therefore, $\dim_{\C} \widehat{H_1G'}< \infty$, as asserted.

For the general case, consider the extension
\begin{equation}
\label{eq:kext}
1\to G' \rightarrow K \rightarrow A  \to 1\, ,
\end{equation}
where $A=K/G'$ is a finitely generated subgroup of $G_{\ab}$, and denote by $(H_1G')_A$ the co-invariants
of the $A$-module $H_1G'$, noting that the canonical projection, $H_1G' \surj  (H_1G')_A$, is $\C G_{\ab}$-linear.

The Hochschild-Serre spectral sequence of \eqref{eq:kext} with trivial $\C$-coefficients (see e.g. \cite[p.171]{B})
provides an exact sequence of finitely generated $\C G_{\ab}$-modules,
\begin{equation}
\label{eq:hshort}
(H_1G')_A \rightarrow H_1K \rightarrow H_1A  \to 0\, .
\end{equation}
By standard commutative algebra (see for instance \cite[Chapter 10]{AM}), the $I$-adic completion of 
\eqref{eq:hshort} is again exact. Since  $\dim_{\C} H_1 A< \infty$, our claim about $\widehat{H_1K}$
follows from the finite-dimensionality of $\widehat{H_1G'}$.
\end{proof}

The next corollary gives an application to Torelli groups. It follows from Theorems \ref{thm:rest} and \ref{thm:findex}.

\begin{corollary}
\label{cor:kt}
Let $K$ be a  subgroup of $T_g$ containing $T_g'$.
The $I$-adic completion of $H_1(K, \C)$ is
finite-dimensional, for $g\ge 4$, where $I\subseteq \C [(T_g)_{\ab}]$ is the
augmentation ideal of the group ring.
\end{corollary}

We will improve Corollary \ref{cor:kt} in the next section, when $K$ contains the Johnson kernel $K_g$.

\section{Characteristic varieties of Torelli groups} \label{sec:dsim}

Guided by the interplay between arithmetic and Torelli groups, coming from \eqref{eq:deftintro},
we will examine now groups whose characteristic varieties possess a natural discrete symmetry.
We will compute the (restricted) characteristic varieties of $T_g$, in degree $1$, when $g\ge 4$,
and deduce that $b_1(K_g)<\infty$.

Our setup in this section is the following. Let 
\begin{equation}
\label{eq:aext}
1\to T \rightarrow \G \stackrel{p}{\rightarrow} D\to 1
\end{equation}
be a group extension, where $T$ is finitely generated and $D$ is an arithmetic subgroup of 
a complex linear algebraic group $S$, defined over $\Q$, simple and 
with $\Q$-rank at least $1$. The motivating examples are the extensions \eqref{eq:deftintro}, 
for $g\ge 3$. Under the above assumptions on $D$, we recall that any finite index subgroup
$D_1\subseteq D$ is Zariski dense in $S$, as follows from Borel's density theorem;
see e.g. \cite[Corollary 5.16]{R2}.

We continue by reviewing a couple of relevant facts related to character tori and characteristic
varieties. Let $G$ be a group, with  character group $\T (G)=\Hom (G_{\ab}, \C^*)$. When $G$ is finitely generated, 
the character torus $\T (G)$ is a linear algebraic group, 
with coordinate ring the group algebra $\C G_{\ab}$. The connected component of 
$1\in \T(G)$ is $\T^0 (G)= \T(G_{\abf})$. 

For the beginning, we need no assumptions on the group extension \eqref{eq:aext}.
The natural $D$-representation in $T_{\ab}$ (respectively $T_{\abf}$) induced by conjugation
canonically extends to $\Z T_{\ab}$ and $\C T_{\ab}$ (respectively to $\Z T_{\abf}$ and $\C T_{\abf}$).
The corresponding left $D$-action on $\T (T)$, by group automorphisms, is denoted by
$d\cdot \rho$, for $d\in D$ and $\rho \in \T (T)$, and is defined by 
$d\cdot \rho (u)=\rho \circ d^{-1} (u)$, for $u\in T_{\ab}$. When $T$ is finitely generated,
the $D$--action on $\T (T)$ is algebraic.

The {\em characteristic varieties} of a group $G$, $\V^i_k (G)$,  are defined for (degree) $i\ge 0$ and
(depth) $k\ge 1$ by 
\begin{equation}
\label{eq:defchar}
\V^i_k (G)= \{ \rho\in \T (G) \mid \dim_{\C} H_i(G, \C_{\rho})\ge k \}\, . 
\end{equation}
Here $\C_{\rho}$ denotes the rank one complex $G$--module given by the change of rings
$\Z G\to \C$, corresponding to $\rho$. When $G$ is finitely generated, it is easy to check that
$\V^i_k (G)$ is Zariski closed in $\T (G)$, for $i\le 1$ and $k\ge 1$. We will be mainly
interested in the (restricted) characteristic variety (in degree $i=1$ and depth $k=1$) $\V (G):=\V^1_1 (G) \cap \T^0 (G)$ 
associated to a finitely generated group $G$. By the above discussion, the restricted (degree $1$) characteristic variety
$\V^1_k (G) \cap \T^0 (G)$ is Zariski closed in the affine connected torus $\T^0 (G)= (\C^*)^{b_1(G)}$, for all $k\ge 1$.

Our starting point in this section is the following.

\begin{lemma}
\label{prop:dsim}
Given an arbitrary group extension \eqref{eq:aext}, 
the characteristic varieties $\V^i_k (T)\subseteq \T (T)$ are $D$-invariant subsets,
for all $i\ge 0$ and $k\ge 1$.
\end{lemma}

\begin{proof}
For $\gamma \in \G$, denote $\gamma$-conjugation by $\iota_{\gamma}: T\isom T$.
For a  ring $R$, with group of units $R^{\times}$, and a group homomorphism,
$\chi: T\to R^{\times}$, the notation $R_{\chi}$ means the $\Z T$-module $R$ associated
to the change of rings $\chi: \Z T\to R$. Note that the pair 
$(\iota_{\gamma}, \id_R)\colon (T, R_{\chi\circ \iota_{\gamma}}) \isom (T, R_{\chi})$,
gives an isomorphism in the category of local systems; see e.g. \cite[III.8]{B}. Hence,
there is an induced isomorphism, $H_*(T, R_{\chi\circ \iota_{\gamma}}) \isom H_*(T, R_{\chi})$.
Our claim follows by taking $R=\C$, and inspecting definition \eqref{eq:defchar}.
\end{proof}

When $T$ is finitely generated, let us consider $\V^1_k (T)\cap \T^0(T)$. This leads us to look at a discrete group $D$
acting linearly on a free, finitely generated abelian group $L$, and examine the
$D$-invariant, Zariski closed subsets $W$ of $\T (L)= \Hom (L, \C^*)$. Besides the trivial case
$W=\T (L)$, we find a lot of $0$-dimensional examples, by taking $W$ to be the subgroup of
$m$-torsion elements of $\T (L)$. This raises a natural question: is there anything else?

To present a first answer, we
need the following notions. A {\em translated subgroup (torus)} is  a subset of 
the character torus $\T= \T (T)$
of the form $t\cdot \bS$, where $t\in \T$ and $\bS \subseteq \T$ is a closed (connected) 
algebraic subgroup. Note that the {\em direction} of the translated subgroup, $\bS$, is
uniquely determined by $t\cdot \bS$. A Zariski closed subset $W\subseteq \T$ is a 
{\em union of translated tori} if  each irreducible component of $W$ is a translated
torus; in other words, $W$ is a finite union of translated subgroups. 

\begin{lemma}
\label{prop:dchar}
Let $L$ be a $D$-module which is finitely generated and free as an abelian group. Assume that $D$ is an
arithmetic subgroup of a simple $\C$-linear algebraic group $S$ defined over $\Q$, with
$\Q-\rank (S)\ge 1$. Suppose also that the $D$-action on $L$ extends to an irreducible, 
rational $S$-representation in $L(\C):= L\otimes \C$. Let $W\subset \T (L)$ be a
$D$-invariant, Zariski closed, proper subset of $\T (L)$. If $W$ is a union of translated tori, then
$W$ is finite.
\end{lemma}

\begin{proof}
We know that each irreducible component of $W$ is of the form $t\cdot \bS$, as above. We have to
show that $\dim (\bS)=0$. To this end, we consider the isotropy group of $t\cdot \bS$, denoted $D_1$;
it is a finite index subgroup of $D$. Note that $\dim (\bS)< \dim (\T (L))$, since $W\ne \T (L)$.

The $D_1$-invariance of $t\cdot \bS$ forces the direction $\bS$ to be $D_1$-invariant as well. 
Therefore, the Lie algebra $T_1 \bS$ is a $D_1$-invariant linear subspace of 
$\Hom (L, \C)=L(\C)^*$. Since $D_1\subseteq S$ is Zariski dense, it follows that the subspace
$T_1 \bS$ is actually $S$-invariant. Hence, $T_1 \bS=0$, due to $S$-irreducibility. 
\end{proof}

To improve Lemma \ref{prop:dchar}, we need a preliminary result.

\begin{lemma}
\label{lem:prel}
Let $L$ be a $D$-module which is finitely generated and free as an abelian group. Then
the subgroup $\OO_t$ of $\T (L)$, generated by the $D$-orbit $D\cdot t$, is finitely generated,
for any $t\in \T (L)$.
\end{lemma}

\begin{proof}
Pick a $\Z$-basis $\{ e_i\}$ of $L$; identify $L$ with $\Z^n$, and $\T (L)$ with $(\C^*)^n$.
For $t=(t_1, \dots, t_n)\in \T (L)$ and $w= \sum_i w_i e_i \in L$, set 
$t^w:= \prod_i t_i^{w_i} \in \C^*$.

The elements of $D\cdot t$ are of the form $(t^{v_1}, \dots, t^{v_n})$, with
$v_i= \sum_j v_{ij} e_j \in L$. Define $u_k^{ij}= \delta_{ik} e_j \in L$, for
$1\le i,j,k \le n$. The equality
\[
\Big( t^{v_1}, \dots, t^{v_n}\Big) = \prod_{1\le i,j\le n} \Big( t^{u_1^{ij}}, \dots, t^{u_n^{ij}}\Big)^{v_{ij}}
\]
gives the desired conclusion.
\end{proof}

The key argument in a stronger version of Lemma \ref{prop:dchar} uses the solution of a conjecture of Lang 
in diophantine geometry, obtained by Laurent in \cite{Lau}. We recall this result. Let $W\subseteq \T (L)$ be 
Zariski closed, and let $\OO \subseteq \T (L)$ be a subgroup of finite rank. Denote by $\TT$ the set of translated subgroups
of $\T (L)$ of the form $\gamma \cdot \bS$, with $\gamma \in \OO$ and such that $\gamma \cdot \bS \subseteq W$,
ordered by inclusion. We extract from \cite{Lau}, Th\' eor\` eme 2, Lemme 3 and Lemme 4, the following result.

\begin{theorem}[Laurent]
\label{thm:langlau}
The set $\TT$ has finitely many maximal elements. Denote them $\gamma_1 \cdot \bS_1, \dots, \gamma_m \cdot \bS_m$.
Furthermore,
\[
W \cap \OO = \bigcup_{i=1}^m \gamma_i \cdot (\bS_i \cap \OO) \, .
\]
\end{theorem}

Now, we are going to show that the $D$-irreducibility of $L$ is inherited by the affine torus $\T (L)$,
in the sense explained below. This proves Theorem \ref{thm:bintro}. 

\begin{theorem}
\label{thm:girr}
Let $L$ be a $D$-module which is finitely generated and free as an abelian group. Assume that $D$ is an
arithmetic subgroup of a simple $\C$-linear algebraic group $S$ defined over $\Q$, with
$\Q-\rank (S)\ge 1$. Suppose also that the $D$-action on $L$ extends to an irreducible, 
rational $S$-representation in $L(\C):= L\otimes \C$. Let $W\subset \T (L)$ be a
$D$-invariant, Zariski closed, proper subset of $\T (L)$. Then $W$ is finite.
\end{theorem}

\begin{proof}
We first claim that $D\cdot t$ is finite, for every $t\in W$.  We will apply Theorem \ref{thm:langlau} to 
the closed subvariety $W\subseteq \T (L)$ and the subgroup $\OO_t \subseteq \T (L)$, which is
finitely generated, by Lemma \ref{lem:prel}. In this way, we obtain the inclusion
\begin{equation}
\label{eq:lang}
W\cap \OO_t \subseteq \bigcup_{i=1}^m \gamma_i \cdot \bS_i \, .
\end{equation}
Denote by $V\subseteq W$ the right-hand side of \eqref{eq:lang}. Since $\OO_t$ is $D$-invariant 
by construction, we infer that $V$ is $D$-invariant as well; this follows from the construction of the $D$--invariant
ordered set $\TT$, whose maximal elements are $\gamma_1 \cdot \bS_1, \dots, \gamma_m \cdot \bS_m$.
Hence, $V$ must be finite, by
Lemma \ref{prop:dchar}. Since clearly $D\cdot t\subseteq W\cap \OO_t$, we obtain our claim.

Supposing that $W$ is infinite, we may find a smooth point $t\in W$ whose tangent space satisfies
$0\ne T_t W \ne T_t \T (L)$. Translating to the origin, we obtain a $D_t$-invariant, proper and
non-trivial linear subspace of $L(\C)^*$, where $D_t$ is the isotropy group of $t$. By the
first step of our proof, $D_t$ has finite index in $D$, which contradicts the $S$-irreducibility
of $L(\C)$.
\end{proof}

We go on by establishing a relation between characteristic and resonance varieties in degree $1$,
for finitely generated groups.

\begin{lemma}
\label{lem:tgcone}
Let $G$ be a finitely generated group.
\begin{enumerate}
\item \label{tc1}
There is a finitely presented group $\overline{G}$, together with a group epimorphism,
$\phi: \overline{G}\surj G$, such that $\phi_{\ab}: \overline{G}_{\ab}\isom G_{\ab}$ is
an isomorphism, inducing identifications, $\V^1_k (\overline{G})\equiv \V^1_k(G)$ and
$\R^1_k (\overline{G})\equiv \R^1_k(G)$, for all $k\ge 1$.
\item \label{tc2}
The tangent cone at $1$ of $\V^1_k(G)$, $TC_1 \V^1_k(G)$, is contained in
$\R^1_k(G)$, for all $k\ge 1$.
\end{enumerate}
\end{lemma}

\begin{proof}
Part \eqref{tc1}. Let $X$ be a classifying space for $G$, having $1$-skeleton equal to a
finite wedge of circles. Denote by 
$\{ \Z G_{\ab} \otimes C_{\bullet} X \stackrel{D_{\bullet}}{\rightarrow}\Z G_{\ab} \otimes C_{\bullet -1} X\}$
the cellular chain complex of the universal abelian cover $X^{\ab}$, where
$\{ C_{\bullet} X \stackrel{d_{\bullet}}{\rightarrow} C_{\bullet -1} X\}$ is
the cellular chain complex of $X$, over $\Z$. Since the ring $\Z G_{\ab}$ is noetherian,
$\im (D_2)$ is generated over $\Z G_{\ab}$ by the images of finitely many $2$-cells,
say $\{ e_1, \dots, e_r \}$. 

Consider now the comultiplication map of the $2$--skeleton, 
$\partial_{X^{(2)}}: H_2(X^{(2)}, \C) \to \wedge^2 H_1(X^{(2)}, \C)$. Pick finitely many $2$-cells, 
say $\{ e'_1, \dots, e'_s \}$, such that the $\partial$-images of the $d_{\bullet}$-cycles 
belonging to $\Z- \spn \{ e'_1, \dots, e'_s \}$ generate $\im (\partial_{X^{(2)}})$ over $\C$.

Let $Y$ be the finite subcomplex of $X^{(2)}$ obtained from $X^{(1)}$ by attaching the cells 
$\{ e_i\}$ and $\{ e'_j \}$. Set $\overline{G}:= \pi_1(Y)$. Attach cells of dimension at least $3$
to $Y$, in order to obtain a classifying space for $\overline{G}$, denoted $\overline{X}$.
Extend the inclusion $Y\hookrightarrow X^{(2)}$ to a map $\overline{X}\to X$, and consider the induced
group epimorphism, $\phi: \overline{G}\surj G$.

By our choice of the $e$-cells, we infer that $\phi$ induces in turn an isomorphism,
$H_1(\overline{G}, R_{\chi\circ \phi})\isom H_1(G, R_{\chi})$, for any group homomorphism,
$\chi:G \to R^{\times}$, when the ring $R$ is commutative. In particular, our claims on $\phi_{\ab}$
and characteristic varieties are thus verified. To check the claim on resonance varieties,
we may replace $\overline{G}$ by $Y$ and $G$ by $X^{(2)}$. From our choice for the $e'$-cells,
we deduce that, upon identifying $\wedge^2 H_1(Y, \C)$ and $\wedge^2 H_1(X^{(2)}, \C)$ 
via $\phi$, we have the equality $\im (\partial_{Y})= \im (\partial_{X^{(2)}})$. By duality,
the cup-product maps $\cup_Y$ and $\cup_{X^{(2)}}$ have isomorphic co-restrictions to the
image, which proves our last claim.

Part \eqref{tc2}. A result of Libgober from \cite{Lib} implies that $TC_1 \V^1_k(G)\subseteq \R^1_k (G)$,
for all $k\ge 1$, when $G$ is finitely presented. By Part \eqref{tc1}, the inclusion still holds for
finitely generated groups.
\end{proof}

Here is our main result in this section, relating arithmetic symmetry and finiteness properties,
which proves Theorem \ref{thm:cintro}. The {\em Alexander polynomial} $\Delta^T$ of a
finitely generated group $T$ appearing in Part \eqref{ta} below is a classical invariant, with roots in
knot theory, defined as follows. Let $A(T)= \Z T_{\abf} \otimes_{\Z T}I$ be the Alexander module.
It is a finitely generated $\Z T_{\abf}$--module, with Fitting ideal $E_1(A(T)) \subseteq \Z T_{\abf}$.
The greatest common divisor of all elements in $E_1(A(T))$, defined up to units of $\Z T_{\abf}$,
is $\Delta^T$.

\begin{corollary}
\label{cor:qkchar}
Assume in \eqref{eq:aext} that $T$ is finitely generated with $\R (T) \subseteq \{ 0\}$, and
$D\subseteq S$ is arithmetic, where the $\C$-linear algebraic group $S$ 
is defined over $\Q$, simple,
with $\Q-\rank (S)\ge 1$. Suppose moreover that the canonical $D$-representation in $T_{\abf}$ 
extends to an irreducible, rational $S$-representation in $T_{\abf}\otimes \C$.
Then the following hold.
\begin{enumerate}
\item \label{tv}
The restricted characteristic varieties $\V^1_k (T)\cap \T^0 (T)$ are finite, for all $k\ge 1$.
\item \label{ti}
The first Betti number of $\pi^{-1}(A)$ is finite, for any subgroup $A\subseteq T_{\abf}$,
where $\pi: T\surj T_{\abf}$ is the canonical projection.
\item \label{ta}
If moreover $b_1(T)>1$, the Alexander polynomial $\Delta^T$ is a non-zero constant $c\in \Z$,
modulo the units of $\Z T_{\abf}$.
\end{enumerate}
\end{corollary}

\begin{proof}
We may clearly assume $b_1(T)\ne 0$. 
We want to use Theorem \ref{thm:girr}, applied to the $D$-module $L=T_{\abf}$ and the
closed subvariety $W=\V^1_k (T)\cap \T^0 (T)$. The $D$-invariance of $W$ follows from 
Lemma \ref{prop:dsim}. The fact that $W\ne \T (L)$ is a consequence of 
our assumption on $\R (T)$, due to Lemma \ref{lem:tgcone}\eqref{tc2}. Hence,
Part \eqref{tv} follows from Theorem \ref{thm:girr}.

Set $K:= \ker (\pi)$. 
A basic result of Dwyer and Fried \cite{DF}, as refined in
\cite[Corollary 6.2]{PS-bns}, says that the finiteness of $\V_1^1 (T) \cap \T^0(T)$ is
equivalent to $\dim_{\C} (K_{\ab})\otimes \C <\infty$, for any finitely generated group $T$.
Consider now the extension $1\to K\to \pi^{-1}(A)\to A\to 1$. A standard application of the
Hochschild-Serre spectral sequence  shows that the first Betti
number of $\pi^{-1}(A)$ is finite, since both $b_1(K)$ and $b_1(A)$ are finite. This
completes the proof of Part \eqref{ti}.

Part \eqref{ta} follows from Part \eqref{tv}, via Corollary 3.2 from \cite{PS-cod}.
\end{proof}

Corollary \ref{cor:qkchar} leads to the following consequences, for Torelli groups.

\begin{corollary}
\label{cor:tor}
Assume $g\ge 4$.
\begin{enumerate}
\item \label{torv}
The intersection $\V^1_k (T_g)\cap \T^0 (T_g)$ is finite, for all $k\ge 1$.
\item \label{tork}
The vector space $H_1(N, \C)$ is finite-dimensional, for any subgroup $N$ of $T_g$
containing the Johnson kernel $K_g$. 
\item \label{tora}
The Alexander polynomial $\Delta^{T_g}$ is a non-zero constant $c\in \Z$, modulo units. 
\end{enumerate}
\end{corollary}

\begin{proof}
The condition on resonance is guaranteed by Theorem \ref{thm:rest}. Since $b_1(T_g)$ is
an increasing function of $g$, for $g\ge 3$, $b_1(T_g)\ge 14$, when $g\ge 3$. 
\end{proof}

\section{Sigma-invariants and the Kahler property} \label{sec:rsigma}

We close by examining groups with natural arithmetic symmetry, at the level of 
Bieri-Neumann-Strebel-Renz invariants. We first review briefly the definitions and
the main properties; for more details and references, see e.g. \cite{PS-bns}.

We start with the {\em Novikov-Sikorav completion} of a finitely generated group $G$, 
with respect to $\chi\in \Hom(G, \RR)$, denoted $\widehat{\Z G}_{-\chi}$. For
$k\in \Z$, let $F_k$ be the abelian subgroup of $\Z G$ generated by the elements
$g\in G$ with $\chi (g)\ge k$. The completion, $\widehat{\Z G}_{-\chi}$, of $\Z G$
with respect to the decreasing filtration $\{ F_k\}_{k\in \Z}$ becomes in a natural way
a ring, containing the group ring $\Z G$. 

The {\em Sigma-invariants} $\Sigma^q(G, \Z)$ are defined for $q\ge 1$ by
\begin{equation}
\label{eq:defsigma}
\Sigma^q(G, \Z)= \{ 0\ne \chi\in H^1(G, \RR)\mid H_{i}(G, \widehat{\Z G}_{-\chi})=0 \, , \forall i\le q \}\, .
\end{equation}
When the group $G$ is of type $FP_k$, the above definition coincides with the one
introduced by Bieri and Renz in \cite{BR}, for $q\le k$. Note that property $FP_1$
simply means finite generation of $G$. The Sigma-invariant $\Sigma^1(G, \Z)$ of a
finitely generated group $G$, denoted $\Sigma (G)$, coincides with the one defined by
Bieri, Neumann and Strebel in \cite{BNS}.

It turns out that $\Sigma^q(G, \Z)$ is an open (possibly empty) conical subset of
$H^1(G, \RR)$, for all $q\le k$, when $G$ is of type $FP_k$. Moreover, in this case 
we have the following fundamental property. Given a group epimorphism, $\nu: G\surj L$, onto
an abelian group, $\ker (\nu)$ is of type $FP_q$, with $q\le k$, if and only if
\begin{equation}
\label{eq:ftest}
\nu^*(H^1(L, \RR)\setminus \{ 0\})\subseteq \Sigma^q(G, \Z)\, .
\end{equation}

Here is the analog of Lemma \ref{prop:dsim} for Sigma-invariants.

\begin{lemma}
\label{prop:rsim}
Let $p:\G \surj D$ be a group epimorphism, with finitely generated kernel $T$. Then
$\Sigma^q(T, \Z)$ is invariant under the canonical action of $D$ on $H^1(T, \RR)$, 
coming from \eqref{eq:aext}, for all $q\ge 1$.
\end{lemma}

\begin{proof}
Novikov-Sikorav completion is functorial, in the following sense. Let $\phi:G\to K$ be
a group homomorphism and $\chi\in \Hom(K, \RR)$. The induced ring homomorphism,
$\phi: \Z G\to \Z K$, clearly preserves the defining filtrations of 
$\widehat{\Z G}_{-\chi\circ \phi}$ and $\widehat{\Z K}_{-\chi}$. Passing to completions, 
$\phi$ extends to a ring homomorphism, 
$\widehat{\phi}: \widehat{\Z G}_{-\chi\circ \phi} \to \widehat{\Z K}_{-\chi}$.

The above remark may be applied to $\gamma$-conjugation, $\phi=\iota_{\gamma}: T\isom T$, 
for any $\gamma\in \G$, and an arbitrary additive character $\chi\in H^1(T, \RR)$. The
pair $(\phi, \widehat{\phi})$ gives then an isomorphism,
$(T, \widehat{\Z T}_{-\chi\circ \phi}) \isom (T, \widehat{\Z T}_{-\chi})$, in the 
category of local systems. Consequently, there is an induced isomorphism,
$H_*(T, \widehat{\Z T}_{-\chi\circ \phi}) \isom H_*(T, \widehat{\Z T}_{-\chi})$.
We infer from \eqref{eq:defsigma} that $\chi\in \Sigma^q(T, \Z)$ if and only if
$\chi\cdot d\in \Sigma^q(T, \Z)$, where $d=p(\gamma)$.
\end{proof}

\begin{remark}
\label{rem:anti}
Note that $-\id \in Sp_g(\Z)$ acts by $-\id$ on $\wedge^3 H/H$. Consequently, 
$-\Sigma (T_g)=\Sigma (T_g)$. This symmetry property of $\Sigma (G)$ about the origin 
does not hold in general.

Note also that, when $-\Sigma (G)=\Sigma (G)$, $\Sigma (G)\ne \emptyset$ if and only if
there is a finitely generated, normal subgroup $N$ of $G$, with infinite abelian quotient $G/N$.
Moreover, in this statement $G/N$ may be replaced by $\Z$. Indeed, assuming $G/N$ to be 
infinite abelian, we infer that $\Sigma (G)\ne \emptyset$, by resorting to \eqref{eq:ftest}.
Conversely, we know that the image of $\Sigma (G)$ in the quotient sphere, 
$(H^1(G, \RR)\setminus \{ 0\})/\RR_{+}$, is open and nonvoid. The density of 
rational points on this sphere \cite[p.451]{BNS} implies then that $\Sigma (G)$ contains
an epimorphism, $\nu:G \surj \Z$. By antipodal symmetry of $\Sigma (G)$ and 
\eqref{eq:ftest} again, $\ker (\nu)$ must be finitely generated.
\end{remark}

Under additional hypotheses, Lemma \ref{prop:rsim} may be used to obtain strong
information on finiteness properties.

\begin{prop}
\label{cor:bnsa}
Let $p:\G \surj D$ be a group epimorphism, with finitely generated kernel $T$.
Assume that $D$ is an arithmetic subgroup of a complex, simple linear algebraic group $S$,
defined over $\Q$, with $\Q-\rank (S)\ge 1$, and the canonical $D$-action on $T_{\abf}$ 
induced by conjugation extends to a non-trivial, irreducible rational representation
of $S$ in $T_{\abf}\otimes \C$. Suppose moreover that
\begin{enumerate}
\item \label{r1}
either $\Sigma(T)$ is a finite disjoint union of finite intersections of open half-spaces
in $H^1(T, \RR)$,
\item \label{r2}
or $\Sigma(T)$ is the complement of a finite union of linear subspaces in $H^1(T, \RR)$.
\end{enumerate}
Then the kernel $K$ of the natural map, $T\surj T_{\abf}$, is finitely generated if and
only if $\Sigma(T)\ne \emptyset$.
\end{prop}

\begin{proof}
According to \eqref{eq:ftest}, finite generation of $K$ is equivalent to 
$\Sigma(T)= H^1(T, \RR)\setminus \{0\}$. Note that our irreducibility assumptions imply
that $b_1(T)>1$. We claim that if $\Sigma(T)$ is a proper, non-void subset of 
$H^1(T, \RR)\setminus \{0\}$, then there is a proper, non-trivial linear subspace
$E\subseteq H^1(T, \RR)$ invariant under the canonical action of a finite index subgroup
$D_0\subseteq D$. Granting the claim, we may use the fact that $D_0$ is Zariski dense
in $S$ to infer that $E\otimes \C\subseteq (T_{\abf}\otimes \C)^*$ is $S$-invariant,
a contradiction. Thus, we only need to verify the above claim, in order to finish the proof.

\eqref{r1} In this case, we know that $\Sigma(T)= \cup_{i=1}^r C_i$, where each $C_i$ is 
a chamber of a non-void, finite hyperplane arrangement $\A_i$ in $H^1(T, \RR)$,
and the union is disjoint. By Lemma \ref{prop:rsim}, there is a finite index subgroup
$D_1\subseteq D$ such that $C_1\cdot d=C_1$, for any $d\in D_1$. Consider the supporting
hyperplanes of $C_1$, that is, the set $\SS_1$ consisting of those hyperplanes 
$E\subseteq H^1(T, \RR)$ with the property that the intersection of $E$ with the boundary
of $C_1$ has non-void interior in $E$. Standard arguments show that $\SS_1$ is a non-void
subset of $\A_1$; see \cite[Chapter V.1]{Bo}. Clearly, $\SS_1$ is $D_1$-invariant, so 
we may choose $D_0$ to be the isotropy group of an element $E\in \SS_1$.

\eqref{r2} In the second case, the complement of $\Sigma(T)$ is the union of a non-void, 
finite arrangement $\A$ of non-trivial, proper linear subspaces of $H^1(T, \RR)$. Again
by Lemma \ref{prop:rsim}, $\A$ is $D$-invariant. Clearly, the isotropy group $D_0$
of $E\in \A$ satisfies the desired conditions.
\end{proof}

Property \eqref{r1} above is verified by $3$-manifold groups (that is, fundamental
groups of compact, connected, differentiable $3$-manifolds), according to \cite{BNS}.
Property \eqref{r2} holds for Kahler groups, as shown by Delzant in \cite{D}. For
Torelli groups, Proposition \ref{cor:bnsa}\eqref{r2} may be improved 
to obtain Theorem \ref{thm:kintro}.

\begin{corollary}
\label{cor:ftor}
If the Torelli group $T_g$ ($g\ge 4$) is a Kahler group,
then the Johnson kernel $K_g$ is finitely generated.
\end{corollary}

\begin{proof}
We have to show that $\Sigma(T_g)\ne \emptyset$, for $g\ge 4$, assuming the Kahler property.
According to \cite{D}, the complement of $\Sigma(T_g)$ is a finite union,
\[
\bigcup_{\alpha} f^*_{\alpha} H^1(C_{\alpha}, \RR)\, ,
\]
coming from irrational pencils on the compact Kahler manifold $M$, if $T_g=\pi_1(M)$.
More precisely, each $f_{\alpha}:M \to C_{\alpha}$ is a holomorphic map onto a smooth 
compact curve with $\chi (C_{\alpha})\le 0$, having connected fibers, and
$\dim_{\RR}  f^*_{\alpha} H^1(C_{\alpha}, \RR)= b_1(C_{\alpha})$. If
$b_1(C_{\alpha})<b_1(T_g)$, for every $\alpha$, then we are done. Otherwise, 
$\chi (C_{\alpha})< 0$, for some $\alpha$, since $b_1(T_g)>2$. We infer from definition
\eqref{eq:defres} that $\R^1_1(C_{\alpha})=H^1(C_{\alpha}, \C)$ and 
$f^*_{\alpha} H^1(C_{\alpha}, \RR) \subseteq \R^1_1(T_g)\cap H^1(T_g, \RR)$, 
which contradicts Theorem \ref{thm:rest}.
\end{proof}

\begin{example}
\label{rem:ell}
We point out that there exist Kahler groups with $\R^1_1(G)=0$ and $\Sigma(G)= \emptyset$.
In particular, \eqref{eq:ftest} implies that in this situation the kernel of 
the canonical projection, $G\surj G_{\abf}$, is not finitely generated. 
At the same time, the condition on $\R^1_1(G)$ implies that there is no group epimorphism,
$G\surj \pi_1(\Sigma_h)$, when the genus $h$ is at least $2$, by the argument in the proof 
of Corollary \ref{cor:ftor}. We give such an
example, inspired by a construction of Beauville \cite[Example 1.8]{Be}.

Let $g$ be a fixed-point free involution of a $1$-connected compact Kahler manifold $E$.
The existence of such an object follows for instance from Serre's result \cite{Se}, 
which guarantees the realizability of finite groups as fundamental groups of smooth, 
projective complex varieties. Let $\Z_2$ act on the Fermat curve 
$F:= \{ x^4+y^4+z^4=0\}\subseteq \C \PP^2$ by $g(x:y:z)=(y:x:z)$. Set 
$M:= F\times E/\Z_2$, where the quotient is taken with respect to the diagonal action,
and $C:=F/\Z_2$. Note that $M$ is a compact Kahler manifold, and $C$ is an elliptic curve.
Set $G:= \pi_1(M)$.

The first projection induces a holomorphic surjection, $f:M\to C$, having connected fibers,
and $4$ multiple fibers (of multiplicity $2$). By Delzant \cite{D}, the subspace $f^*H^1(C, \RR)$
is contained in the complement of $\Sigma(G)$. Since $f^{\bullet}:H^{\bullet}C \to H^{\bullet}M$
may be identified with the inclusion of fixed points, 
$(H^{\bullet}F)^{\Z_2} \hookrightarrow (H^{\bullet}F \otimes H^{\bullet}E)^{\Z_2}$, and
$b_1(E)=0$, we infer that $f$ induces in cohomology an isomorphism in degree $1$ and a
monomorphism in degree $2$. It follows that $\Sigma(G)= \emptyset$ and $\R^1_1(G)=0$,
as asserted.

By a similar construction, we may exhibit examples of Kahler groups with arbitrary (non-zero)
even first Betti number, having the property that $\R^1_1(G)=0$ and 
$\Sigma(G) \ne H^1(G, \RR)\setminus \{ 0\}$. As before, these conditions imply that the kernel
of the canonical projection, $G\surj G_{\abf}$, is not finitely generated, and there is no group
epimorphism, $G\surj \pi_1(\Sigma_h)$, for $h\ge 2$. 
\end{example}

\begin{remark}
\label{rem:farb}
In the proof of Corollary \ref{cor:ftor}, our strategy involves two steps. Firstly, group surjections
$T_g\surj \pi_1(\Sigma_h)$ with $h\ge 2$ are ruled out with the aid of our result on resonance, from
Theorem \ref{thm:rest}. Secondly, we use the symplectic symmetry of Sigma-invariants to prove 
non-existence of group surjections onto orbifold fundamental groups in genus $1$ (or, elliptic pencils with
multiple fibers, in the geometric language from \cite{D}), $T_g\surj \pi_1^{\orb}(\Sigma_1)$.
Example \ref{rem:ell} shows that the second step is needed for the proof of Corollary \ref{cor:ftor}.
\end{remark}

\begin{ack}
We are grateful to Alex Suciu, for useful discussions at an early stage of this work.
Thanks are also due to the referee, whose suggestions helped us to improve the exposition.
\end{ack}

\newcommand{\arxiv}[1]
{\texttt{\href{http://arxiv.org/abs/#1}{arXiv:#1}}}

\bibliographystyle{amsplain}

\begin{thebibliography}{00}

\bibitem{Ak} T.~Akita,
{\em Homological infiniteness of Torelli groups},
Topology \textbf{40} (2001), no.~2, 213--221.

\bibitem{And} S.~Andreadakis, 
{\em On the automorphisms of free groups and free 
nilpotent groups}, Proc. London Math. Soc. \textbf{15} 
(1965), no.~15, 239--268.  

\bibitem{A} D.~Arapura, 
{\em Geometry of cohomology support loci for local systems  {\rm I}}, 
J. Algebraic Geom. \textbf{6} (1997), no.~3, 563--597.  

\bibitem{AM} M.~F.~Atiyah, I.~G.~MacDonald,
{\em Introduction to commutative algebra}, Addison-Wesley,
Reading, Massachussetts, 1969.

\bibitem{Be} A.~Beauville, 
{\em Annulation du $H^1$ pour les fibr\'{e}s en droites plats}, 
in: {\em Complex algebraic varieties (Bayreuth, 1990)}, pp. 1--15,
Lecture Notes in Math., vol.~1507, Springer, Berlin, 1992. 

\bibitem{BBM} M.~Bestvina, K.-U.~Bux, D.~Margalit,
{\em The dimension of the Torelli group},
J. Amer. Math. Soc. \textbf{23} (2010), no.~1, 61--105.

\bibitem{BNS} R.~Bieri, W.~Neumann, R.~Strebel, 
{\em A geometric invariant of discrete groups}, 
Invent. Math. \textbf{90} (1987), no.~3, 451--477. 

\bibitem{BR} R.~Bieri, B.~Renz,
{\em Valuations on free resolutions and higher geometric 
invariants of groups}, Comment. Math. Helvetici 
\textbf{63} (1988), no.~3, 464--497.

\bibitem{BF} D.~Biss, B.~Farb,
{\em $\mathcal{K}_g$ is not finitely generated},
Invent. Math. \textbf{163} (2006), no.~1, 213--226;
erratum \textbf{178} (2009), no.~1, 229.

\bibitem{Bo} N.~Bourbaki,
{\em Groupes et alg\` ebres de Lie, Chapitres 4-6},
Hermann, Paris, 1968.

\bibitem{B}  K.~S.~Brown,
{\em Cohomology of groups}, Grad. Texts in Math., 
vol.~87, Springer-Verlag, New York-Berlin, 1982.

\bibitem{CT} J.~A.~Carlson, D.~Toledo,
{\em Quadratic presentations and nilpotent K\" ahler groups},
J. Geom. Anal. \textbf{5} (1995), no.~3, 359--377; erratum
\textbf{7} (1997), no.~3, 511--514.

\bibitem{DGMS}  P.~Deligne, P.~Griffiths, J.~Morgan, D.~Sullivan,
{\em Real homotopy theory of {K}\"{a}hler manifolds},
Invent. Math. \textbf{29} (1975), no.~3, 245--274.

\bibitem{D} T.~Delzant,
{\em L'invariant de {B}ieri {N}eumann {S}trebel des groupes 
fondamentaux des vari\'{e}t\'{e}s k\"{a}hl\'{e}riennes}, 
Math. Annalen \textbf{348} (2010), 119--125.

\bibitem{PS-cod} A.~Dimca, S.~Papadima, A.~Suciu,
{\em Alexander polynomials: Essential variables and multiplicities}, 
Int. Math. Research Notices vol.~\textbf{2008} (2008), no.~3, 
Art. ID rnm119, 36 pp.

\bibitem{DPS} A.~Dimca, S.~Papadima, A.~Suciu,
{\em Topology and geometry of cohomology jump loci}, 
Duke Math. Journal \textbf{148} (2009), no.~3, 405--457.

\bibitem{DF} W.~G.~Dwyer, D.~Fried,
{\em Homology of free abelian covers. \textup{I}},
Bull. London Math. Soc. \textbf{19} (1987), no.~4, 350--352.

\bibitem{E} D.~Eisenbud,
{\em Commutative algebra with a view towards algebraic geometry},
Grad. Texts in Math., vol.~150, Springer-Verlag, New~York, 1995. 

\bibitem{F} B.~Farb,
{\em Some problems on mapping class groups and moduli space}, in:
{\em Problems on mapping class groups and related topics}, pp. 11--55,
Proc. Sympos. Pure Math., vol.~74, Amer. Math. Soc., Providence, RI, 2006.

\bibitem{H-msri} R.~Hain, 
{\em Torelli groups and geometry of moduli spaces of curves}, in:
{\em Current topics in complex algebraic geometry (Berkeley, 1992/1993)},
pp.~97--143, MSRI Publ., vol.~28, 1995.

\bibitem{H} R.~Hain,
{\em Infinitesimal presentations of the {T}orelli groups},
J. Amer. Math. Soc. \textbf{10} (1997), no.~3, 597--651.

\bibitem{H-sur}  R.~Hain,
{\em Finiteness and {T}orelli spaces}, in:
{\em Problems on mapping class groups and related topics}, pp. 57--70,
Proc. Sympos. Pure Math., vol.~74, Amer. Math. Soc., Providence, RI, 2006.

\bibitem{HL} J.~E.~Humphreys,
{\em Introduction to Lie algebras and representation theory},
Grad. Texts in Math., vol.~9, Springer-Verlag, New York, 1972.

\bibitem{HG} J.~E.~Humphreys,
{\em Linear algebraic groups},
Grad. Texts in Math., vol.~21, Springer-Verlag, New York, 1975.

\bibitem{J0} D.~Johnson,
{\em An abelian quotient of the mapping class group $\TT_g$}, 
Math. Ann. \textbf{249} (1980), 225--242.

\bibitem{J1} D.~Johnson,
{\em The structure of the Torelli group \textup{I}: A finite set 
of generators for $\TT$}, 
Ann. of Math. \textbf{118} (1983), 423--442.

\bibitem{J4} D.~Johnson,
{\em A survey of the Torelli group},
Contemp. Math. \textbf{20} (1983), 165--179.

\bibitem{J2} D.~Johnson,
{\em The structure of the Torelli group \textup{II}: A characterization
of the group generated  by twists on bounding curves}, 
Topology \textbf{24} (1985), no.~2, 113--126.


\bibitem{J3} D.~Johnson,
{\em The structure of the Torelli group \textup{III}: The
abelianization of $\TT$}, 
Topology \textbf{24} (1985), no.~2, 127--144.

\bibitem{L} L.~A.~Lambe,
{\em Two exact sequences in rational homotopy theory relating cup
products and commutators},
Proc. Amer. Math. Soc. \textbf{96} (1986), no.~2, 360--364.

\bibitem{Lau} M.~Laurent,
{\em Equations diophantiennes exponentielles},
Invent. Math. \textbf{78} (1984), no.~2, 299--327.

\bibitem{Lib}  A.~Libgober,
{\em First order deformations for rank one local systems 
with a non-vanishing cohomology}, Topology Appl. 
\textbf{118} (2002), no.~1-2, 159--168. 

\bibitem{M}  A.~Macinic,
{\em Cohomology rings and formality properties of nilpotent groups},
J. Pure Appl. Algebra \textbf{214} (2010), 1818--1826.

\bibitem{Mas}  W.~S.~Massey,
{\em Completion of link modules}, Duke Math. J.
\textbf{47} (1980), no.~2, 399--420.  

\bibitem{MM} D.~McCullough, A.~Miller,
{\em The genus $2$ Torelli group is not finitely generated},
Topology Appl. \textbf{22} (1986), no.~1, 43--49.

\bibitem{Mo} S.~Morita,
{\em Casson's invariant for homology $3$--spheres and characteristic classes of surface bundles} I,
Topology \textbf{28} (1989), no.~3, 305--323.

\bibitem{PS} S.~Papadima, A.~Suciu,
{\em Chen {L}ie algebras}, Int. Math. Res. Notices 
\textbf{2004}, no.~21, 1057--1086. 

\bibitem{PS-bns} S.~Papadima, A.~Suciu,
{\em Bieri--{N}eumann--{S}trebel--{R}enz invariants and 
homology jumping loci},  
Proc. London Math. Soc. \textbf{100} 
(2010), no.~3, 795--834.

\bibitem{PS10} S.~Papadima, A.~Suciu,
{\em Homological finiteness in the Johnson filtration of  
the automorphism group of a free group}, \arxiv{1011.5292},
to appear in J. Topol., {\tt doi:10.1112/jtopol/jts023}


\bibitem{Po} H.~Poincar\' e,
{\em Cinqui\` eme compl\' ement \` a l'analysis situs},
Rend. Circ. Mat. Palermo \textbf{18} (1904), 45--110.

\bibitem{Q}  D.~Quillen,
{\em Rational homotopy theory}, Ann. of Math. 
\textbf{90} (1969), 205--295.

\bibitem{R2} M.~S.~Raghunathan,
{\em Discrete subgroups of Lie groups},
Ergebnisse der Math., vol.~68, Springer-Verlag, Berlin, 1972.

\bibitem{Se} J.-P.~Serre,
{\em Sur la topologie des vari\'{e}t\'{e}s alg\'{e}briques en
charact\'{e}ristique $p$}, in: {\em Symposium internacional de 
topolog\'{i}a algebraica} (Mexico City, 1958), pp.~24--53. 

\bibitem{S} D.~Sullivan,
{\em On the intersection ring of compact three manifolds},
Topology \textbf{14} (1975), 275--277.

\bibitem{Su}  D.~Sullivan,
{\em Infinitesimal computations in topology}, Inst. Hautes 
\'{E}tudes Sci. Publ. Math. \textbf{47} (1977), 269--331.


\end{thebibliography}

\end{document}